\documentclass[11pt]{article}
\usepackage[colorlinks=true, linkcolor=black]{hyperref}

\usepackage{amsmath, amssymb, amsthm}
\usepackage{xcolor,todonotes}
\usepackage{graphicx}
\usepackage{tikz,hyperref}
\usepackage{stmaryrd} 
\usepackage[linesnumbered,ruled,noline]{algorithm2e} 
\usepackage{mathrsfs,cite}

\usepackage{thmtools, thm-restate}
\usepackage{geometry}
\usepackage{enumerate} 
\usepackage{xcolor} 
\usepackage{amsmath} 
\usepackage{nccmath} 
\usepackage{amssymb} 
\usepackage{amsthm}

\usepackage{cite,authblk}
\usepackage{titleref} 
\usepackage{caption}
\usepackage{tikz}
\usepackage{mathrsfs} 
\usepackage{longtable}
\usepackage{color}

\newtheorem{theorem}{Theorem}[section]
\newtheorem{lemma}[theorem]{Lemma}
\newtheorem{corollary}[theorem]{Corollary}
\newtheorem{proposition}[theorem]{Proposition}
\newtheorem{example}[theorem]{Example}
\newtheorem{remark}[theorem]{Remark}
\newtheorem{definition}[theorem]{Definition}
\newtheorem{conjecture}[theorem]{Conjecture}

\DeclareMathOperator {\fA}{\mathfrak{A}}
\DeclareMathOperator {\fAo}{\mathfrak{A}_0}
\DeclareMathOperator {\fAob}{\mathfrak{A}^b_0}

\DeclareMathOperator {\fB}{\mathfrak{B}}
\DeclareMathOperator {\fC}{\mathfrak{C}}

\DeclareMathOperator {\bA}{\mathbf{A}}
\DeclareMathOperator {\bC}{\mathbf{C}}
\DeclareMathOperator {\bB}{\mathbf{B}}

\DeclareMathOperator {\id}{Id}

\DeclareMathOperator {\Pol}{Pol}
\DeclareMathOperator {\pol}{Pol}
\DeclareMathOperator {\NSP}{NSP}
\DeclareMathOperator {\Nsp}{NSP}

\DeclareMathOperator {\csp}{CSP}

\DeclareMathOperator {\Csp}{CSP}
\DeclareMathOperator {\nsp}{NSP}

\DeclareMathOperator {\ra}{RA}

\DeclareMathOperator {\bfA}{\mathbf{A}}

\DeclareMathOperator {\kC}{CON}

\DeclareMathOperator{\Meta}{Meta}
\DeclareMathOperator{\rbcp}{PN}
\DeclareMathOperator{\fra}{FRA}
\DeclareMathOperator{\pc}{PC}
\DeclareMathOperator{\rep}{Rep}

\DeclareMathOperator {\Aut}{Aut}
\DeclareMathOperator {\rra}{RRA}

\DeclareMathOperator {\dom}{Dom}

\DeclareMathOperator {\var}{Var}

\makeatletter
\def\widebreve#1{\mathop{\vbox{\m@th\ialign{##\crcr\noalign{\kern3\p@}%
      \brevefill\crcr\noalign{\kern3\p@\nointerlineskip}%
      $\hfil\displaystyle{#1}\hfil$\crcr}}}\limits}

\def\brevefill{$\scriptscriptstyle\m@th \setbox\z@\hbox{$\scriptscriptstyle\braceld$}%
  \bracelu\leaders\vrule \@height\ht\z@ \@depth\z@\hfill\braceru$}
\makeatletter

\title{Network Satisfaction Problems Solved by \texorpdfstring{$k$}{k}-Consistency}

\author{Manuel Bodirsky\footnote{Manuel Bodirsky has been funded by the European Research Council (Project POCOCOP, ERC Synergy Grant 101071674) and the DFG (Project FinHom, Grant 467967530). Views and opinions expressed are however those of the authors only and do not necessarily reflect those of the European Union or the European Research Council Executive Agency. Neither the European Union nor the granting authority can be held responsible for them.} 

 Manuel.Bodirsky@tu-dresden.de \and Simon Kn\"auer 
 
 simon.knaeuer@web.de}
\affil{Institut f\"{u}r Algebra, TU Dresden, 01062 Dresden, Germany}

\begin{document}

\maketitle

\begin{abstract}
    We show that the problem of deciding for a given finite relation algebra ${\bf A}$ whether the network satisfaction problem for ${\bf A}$ can be solved by the $k$-consistency procedure, for some $k \in {\mathbb N}$, is undecidable. For the important class of finite relation algebras ${\bf A}$ with a normal representation, however, the decidability of this problem remains open. We show that if ${\bf A}$ is symmetric and has a flexible atom, then the question whether $\nsp({\bf A})$ can be solved by $k$-consistency, for some $k \in {\mathbb N}$, is decidable (even in polynomial time in the number of atoms of ${\bf A}$).  
    This result follows from a more general sufficient condition for the correctness of the $k$-consistency procedure for finite symmetric relation algebras. In our proof we make use of a result of Alexandr Kazda about finite binary conservative structures. 
\end{abstract}

\section{Introduction}
Many computational problems in qualitative temporal and spatial reasoning can be phrased as \emph{network satisfaction problems (NSPs)} for \emph{finite relation algebras}.
Such a network consists of a finite set of nodes, and a labelling of  pairs of nodes by elements of the relation algebra. In applications, such a network models some partial (and potentially inconsistent)  knowledge that we have about some temporal or spatial configuration. 
The computational task is to replace the labels by \emph{atoms} of the relation algebra such that the resulting network has an embedding into a representation of the relation algebra. In applications, this embedding provides a witness that the input configuration is consistent (a formal definition of relation algebras, representations, and the network satisfaction problem can be found in Section~\ref{sect:rel-alg}). The computational complexity of the network satisfaction problem depends on the fixed finite relation algebra, and is of central interest in the mentioned application areas. 
Relation algebras have been studied since the 40's with famous contributions of Tarski~\cite{TarskiRelationAlgebras}, Lyndon~\cite{LyndonRelationAlgebras}, McKenzie~\cite{mckenzie1966,McKenzieRepIntRA}, and many others, 
with renewed interest since the 90s~\cite{HirschHodkinsonRepresentability,HirschHodkinsonStrongly,Hirsch,HirschHodkinson,Duentsch,BodirskyRamics,BodirskyKnaeAAAI}.

One of the most prominent algorithms for solving NSPs in polynomial time is the so-called \emph{path consistency procedure}. The path consistency procedure has a natural generalisation to the \emph{$k$-consistency procedure}, for some fixed $k \geq 3$. Such consistency algorithms have a number of advantages: e.g., they run in polynomial time, and they are \emph{one-sided correct} (sometimes called \emph{sound}), i.e., if they reject an instance, then we can be sure that the instance is unsatisfiable. Because of these properties, consistency algorithms can be used to prune the search space in exhaustive approaches that are used if the network consistency problem is NP-complete. 
The question for what temporal and spatial reasoning problems the $k$-consistency procedure provides a necessary and sufficient condition for satisfiability is among the most important research problems in the area~\cite{NebelRenzSurvey,Qualitative-Survey}. 
The analogous problem for so-called \emph{constraint satisfaction problems (CSPs)} 
was posed by Feder and Vardi~\cite{FederVardi} and 
has been solved for finite-domain CSPs by Barto and Kozik~\cite{BoundedWidthJournal}. 
Their result also shows that for a given finite-domain template, the question whether the corresponding CSP can be solved by the $k$-consistency procedure can be decided algorithmically (even in polynomial time if the given template is a core, e.g., if the template contains the relation $\{a\}$ for every domain element $a$; also see~\cite{MetaChenLarose}). 

In contrast, we show that there is no algorithm that decides for a given finite relation algebra ${\bf A}$ whether $\nsp(\bf A)$ can be solved by the $k$-consistency procedure, for some $k \in {\mathbb N}$. The question is also undecidable for every fixed $k \geq 3$; in particular, there is no algorithm that decides whether $\nsp(\bA)$ can be solved by the path consistency procedure (Theorem~\ref{thm:RBCP-undecidable}).  Our proof relies on results of Hirsch~\cite{Hirsch-Undecidable} and Hirsch and Hodkinson~\cite{HirschHodkinsonRepresentability}. The proof also shows that
Hirsch's \emph{Really Big Complexity Problem (RBCP;~\cite{Hirsch})} is undecidable. The RBCP asks for a description of those finite relation algebras ${\bf A}$ whose NSP can be solved in polynomial time. 

Many of the classic examples of relation algebras that are used in temporal and spatial reasoning, such as the point algebra, Allen's Interval Algebra, RCC5, RCC8, have so-called \emph{normal representations}, which are representations that are particularly well-behaved from a model theory perspective~\cite{Hirsch,Qualitative-Survey,BodirskyRamics}. 
The importance of normal representations combined with our negative results for general finite relation algebras 
prompts the question whether solvability of the NSP by the $k$-consistency procedure can at least be characterised for relation algebras $\bf A$ with a normal representation. Our main result is a sufficient condition that implies that $\nsp({\bf A})$ can be solved by the $k$-consistency procedure (Theorem~\ref{theo:Datatractability}). The condition can be checked algorithmically for a given ${\bf A}$. 
Moreover, for symmetric relation algebras with a flexible atom, which form a large subclass of the class of relation algebras with a normal representation, 
our condition provides a necessary and sufficient criterion for solvability by $k$-consistency  (Theorem~\ref{theo:Datalogcharact}). 
We prove that the NSP for every symmetric relation algebra with a flexible atom that cannot be solved by the $k$-consistency procedure is already NP-complete. 
Finally, for symmetric relation algebras with a flexible atom our tractability condition can even be checked in polynomial time for a given relation algebra ${\bf A}$ (Theorem~\ref{thm:meta}). 


In our proof, we exploit a connection between the NSP for relation algebras ${\bf A}$ with a normal representation and finite-domain constraint satisfaction problems. In a next step, this allows us to use strong results for CSPs over finite domains. 
There are similarities between
the fact that the set of relations of a representation of ${\bf A}$ is closed under taking unions on the one hand,
and so-called \emph{conservative finite-domain CSPs}~\cite{Conservative,Barto-Conservative, BulatovConservative,Bulatov-Conservative-Revisited} on the other hand; in a conservative CSP the set of allowed constraints in instances of the CSP contains all unary relations. 
The complexity of conservative CSPs has been classified long before the solution of the Feder-Vardi Dichotomy Conjecture~\cite{FederVardi,BulatovFVConjecture,ZhukFVConjecture,Zhuk20}. 
Moreover, there are particularly elegant descriptions of when a finite-domain conservative CSP {can be solved by the  $k$-consistency procedure for some $k\in \mathbb{N}$ (see, e.g., Theorem 2.17 in \cite{BulatovConservative}).}
Our approach is to turn the similarities into a formal correspondence so that we can use these results for finite-domain conservative CSPs to prove that $k$-consistency solves $\nsp({\bf A})$. A key ingredient here is  a contribution of Kazda~\cite{Kazda2015} about conservative \emph{binary} CSPs.

\section{Preliminaries}\label{sec:prelim}
A \emph{signature} $\tau$ is a set of {function or relation symbols each of which has an associated finite \emph{arity} $k\in \mathbb{N}$. A $\tau$-structure $\fA$ consists of a set $A$ together with a function $f^{\fA}\colon A^k\rightarrow A$ 
for every function symbol $f \in \tau$ of arity $k$
 and a relation $R^{\fA}\subseteq A^k$ for every relation symbol $R \in \tau$ of arity $k$.} The set $A$ is called the \emph{domain of $\fA$}. 
{Let $\fA $ and $\fB$ be $\tau$-structures. The \emph{(direct) product} $\fC=\fA \times \fB$ is the $\tau$-structure where \begin{itemize}
\item $A\times B$ is the domain of $\fC$;
\item for every relation symbol $Q$ of arity $n \in \mathbb{N}$ and every tuple $((a_1, b_1),\ldots,(a_n, b_n))\in (A\times B )^n$, we have that $((a_1, b_1),\ldots,(a_n, b_n))\in Q^\mathfrak{C}$ if and only if   $(a_1,\ldots,a_n)\in Q^{\fA}$ and $(b_1,\ldots,b_n)\in Q^{\fB}$;
\item for every function symbol $Q$ of arity $n \in \mathbb{N}$ and every tuple \\$((a_1, b_1),\ldots,(a_n, b_n))\in (A\times B )^n$, we have that $$ Q^{\mathfrak{C}} ((a_1, b_1),\ldots,(a_n, b_n)) := (Q^{\fA}(a_1,\ldots,a_n), Q^{\fB}(b_1,\ldots,b_n) ).$$
\end{itemize}
We denote the (direct) product $\fA\times \fA$ by $\fA^2$.  
The $k$-fold product $\fA \times \cdots \times \fA$ 
is defined analogously and denoted by $\fA^k$.}
Structures with a signature that only contains function symbols are called \emph{algebras} and structures with purely relational signature are called \emph{relational structures}. Since we do not deal with signatures of mixed type in this article, we will  use the term structure for relational structures only.

\subsection{Relation Algebras}
\label{sect:rel-alg}
{\emph{Relation algebras} are particular algebras; in this section we recall their definition and state some of their basic properties. 
We introduce \emph{proper} relation algebras, move on to abstract relation algebras,} and finally define  representations of relation algebras. 
For an introduction to relation algebras we recommend the textbook by Maddux  \cite{Maddux2006-dp}. 


Proper relation algebras are algebras 
whose domain is a set of binary relations over a common domain, {and which are equipped} with certain operations on binary relations. 

\begin{definition}
	Let $D$ be a set and $\mathcal{R}$ a set of binary relations over $D$ such that \\$(\mathcal{R}; \cup,  \bar{\phantom{o}}, 0,1,\id,  \Breve{\phantom{o}}, \circ ) $ is an algebra 
	with operations defined as follows:\begin{enumerate}	
	\item $0:= \emptyset$,
		\item $1:= \bigcup \mathcal{R}$,
		\item $\id:= \{(x,x) \mid x\in D \} $,
		\item $a \cup b := \{(x,y)  \mid (x,y)\in a \vee (x,y)\in b \} $,
		\item $\bar{a}:= 1 \setminus a $,
		
		\item $\Breve{a} := \{(x,y) \mid (y,x)\in a  \}$,
		\item $a\circ b := \{(x,z) \mid \exists y \in D: (x,y)\in a  \textup{~ and~ } (y,z)\in b \}$,

	\end{enumerate}
for $a,b \in \mathcal{R}$.
Then $(\mathcal{R}; \cup,  \bar{\phantom{o}}, 0,1,\id,  \Breve{\phantom{o}}, \circ ) $ is called a \emph{proper relation algebra.}
	
\end{definition}

The class of all proper relation algebras is denoted by PA. 
Abstract relation algebras are a generalisation of proper relation algebras where the domain does not need to be a set of binary relations. 

\begin{definition}\label{def: abRA} An \emph{(abstract) relation algebra} $\bf A$ is an algebra with domain $A$ and signature $\{\cup, \bar{\phantom{o}}, 0,1,\id,  \Breve{\phantom{o}}, \circ \}$ such that
	\begin{enumerate}
		\item the structure $(A; \cup, \cap, \bar{\phantom{o}}, 0,1)$, with $\cap$ defined by $x\cap y := \overline{(\bar{x}\cup \bar{y})}$, is a Boolean algebra,
		\item $\circ$ is an associative binary operation on A, called \emph{composition},
		\item for all $a,b,c,\in A$: $(a\cup b)\circ c= (a\circ c) \cup (b\circ c)$,
		\item for all $a\in A$: $a\circ \id =a$,
		\item for all $a\in A$: $\Breve{\Breve{a} } = a$,
		\item for all $a,b\in A$: $\Breve{x}=\Breve{a}\cup \Breve {b}$ where $x:=a\cup b$,
		\item for all $a,b\in A$: $\Breve{x}=\Breve{b} \circ \Breve{a}$ where $x:=a\circ b$,
			\item for all $a,b,c \in A$: $\bar{b} \cup \big (\Breve{a} \circ \overline{(a\circ b)  } \big ) =\bar{b}$.
			\end{enumerate}
	\end{definition}

	We denote the class of all relation algebras by RA. 
	Let  $\mathbf{A}=(A; \cup,   \bar{\phantom{o}}, 0,1,\id,  \Breve{\phantom{o}}, \circ ) $ be a relation algebra. 
	Note that $\mathbf{A}$ also satisfies $\id \circ a = a$ for all $a \in A$, because it can be easily deduced from item 4, item 5, and item 7 in Definition~\ref{def: abRA}. 
	By definition,  $(A; \cup,\cap, \bar{\phantom{o}}, 0,1)$ is a Boolean algebra and therefore induces a partial order $\leq$ on $A$, which is defined by
	$x \leq y :\Leftrightarrow x \cup y = y$.  
Note that for proper relation algebras this ordering coincides with the set-inclusion order.
The minimal elements of this order in $A\setminus\{0\}$ are called \emph{atoms}. The set of atoms of $\bf A$ is denoted by $A_0$. Note that for the finite Boolean algebra $(A; \cup,\cap ,\bar{\phantom{o}}, 0,1 ) $ each element $a\in A$ can be uniquely represented as the union $\cup$ (or ``join'') of elements from a subset of $A_0$. We will often use this  fact and directly denote elements of the relation algebra $\bA$ by subsets of  $A_0$. 

By item 3.\ in Definition~\ref{def: abRA} the values of the composition operation $\circ$ in $\bfA$ are completely determined by the values of $\circ$ on $A_0$. This means that for a finite relation algebra the operation $\circ$  can be represented by a multiplication table for the atoms $A_0$.  

An algebra with signature $\tau=\{\cup,   \bar{\phantom{o}}, 0,1,\id,  \Breve{\phantom{o}}, \circ\}$ with corresponding arities $2$, $1$, $0$, $0$, $0$, $1$, and $2$ that is isomorphic to some proper relation algebra is called \emph{representable}. The class of representable relation algebras is denoted by RRA. 
Since every proper relation algebra and therefore also every representable relation algebra satisfies the axioms from the previous definition we have $\text{PA} \subseteq \text{RRA} \subseteq \text{RA}$.
A classical result of Lyndon~\cite{LyndonRelationAlgebras} states that there exist finite relation algebras $\bA \in \ra$ that are not representable; so the inclusions above are  proper. 
If a relation algebra $\bA$ is representable then the  isomorphism to a proper relation algebra is usually called a \emph{representation of $\bf A$}. 

We will be interested in the model-theoretic behavior of sets of relations which form the domain of a proper relation algebra, and therefore  consider  relational structures 
whose relations are precisely 
the relations of a proper relation algebra.
If the set of relations of a relational structure $\fB$ forms a proper relation algebra that is a representation of some abstract relation algebra ${\bf A}$, then it will be convenient to also call
$\fB$ a representation of $\bf A$.




\begin{definition}\label{def: rep}
	Let $\mathbf{A}\in \ra$. A \emph{representation of $\mathbf{A}$} is a  relational structure $\mathfrak{B} $ such that 
	\begin{itemize}
		\item $\mathfrak{B}$ is an $A$-structure, i.e., the elements of $A$ are {binary} relation symbols of ${\mathfrak B}$; 
		\item The map $a \mapsto a^\mathfrak{B}$ is an isomorphism between the abstract relation algebra $\mathbf{A}$ and the proper relation algebra
		$(\mathcal{R}; \cup,   \bar{\phantom{o}}, 0,1,\id,  \Breve{\phantom{o}}, \circ ) $ with domain $\mathcal{R}:=\{a^{\fB}\mid a\in A\}$.
		
	\end{itemize}
\end{definition}

Recall that the set of atoms of a relation algebra $\mathbf{A}=(A; \cup,   \bar{\phantom{o}}, 0,1,\id,  \Breve{\phantom{o}}, \circ ) $ is denoted by  $A_0$. The following definitions are crucial for this article.

\begin{definition}
	
A tuple $(x,y,z) \in (A_0)^3$ is called an \emph{allowed triple (of $\bA$)} if $z \leq x \circ y$. 
Otherwise, $(x,y,z)$ is called a \emph{forbidden triple (of $\bA$)}; {in this case $\overline{z}\cup  \overline{x\circ y}=1 $.}
We say that a relational $A$-structure $\mathfrak{ B}$ \emph{induces a forbidden triple (from $\mathbf{A}$)} if there exist $b_1,b_2, b_3\in B$ and $(x,y,z) \in (A_0)^3$ such that $x(b_1,b_2) , y(b_2,b_3)$ and $z(b_1,b_3)$ hold in $\fB$ and $(x,y,z)$ is a forbidden triple of $\bA$. 
\end{definition}
Note that a representation of $\bf A$ by definition does not induce a forbidden triple.
	A relation $R\subseteq A^3$ is called \emph{totally symmetric} 
 if for every bijection $\pi \colon \{1,2,3\}\rightarrow \{1,2,3\}$  we have
	$$(a_1,a_2,a_3)\in R ~\Rightarrow~ (a_{\pi(1)},a_{\pi(2)},a_{\pi(3)})\in R.$$
	The following is an immediate consequence of the definition of allowed triples.

\begin{remark}
 The set of allowed triples of a symmetric  relation algebra $\bA$ is totally symmetric.
 \end{remark}

\subsection{The Network Satisfaction Problem}
In this section we present computational decision problems associated with relation algebras. {We first introduce the inputs to} these decision problems, so-called $\mathbf{A}$-networks.

	\begin{definition}\label{defi:Netw}Let $\mathbf{A}$ be a relation algebra. An  \emph{$\mathbf{A}$-network $(V;f)$} is a finite set $V$ together with a function $f\colon V^2 \rightarrow A$.  
		An $\mathbf{A}$-network $(V;f)$ is \emph{satisfiable in a representation} 
		$\mathfrak{B}$ of $\bf A$ if there exists an assignment $s\colon V \rightarrow B$ such that for all $(x,y) \in E$ the following holds: $$ (s(x), s(y))\in f(x,y)^\mathfrak{B}. $$
		An $\mathbf{A}$-network $(V;f)$ is \emph{satisfiable} if there exists a representation $\mathfrak{B}$ of $\mathbf{A}$ such that $(V;f)$ is satisfiable in $\mathfrak{B}$.\end{definition}
With these notions we can define the network satisfaction problem.

\begin{definition}\label{defi:nsp}
	The \emph{(general) network satisfaction problem} for a finite  relation algebra $\mathbf{A}$, denoted by $\NSP(\bf A)$, is the problem of deciding whether a given  $\mathbf{A}$-network is satisfiable.\end{definition}

In the following we  
assume that for an $\mathbf{A}$-network $(V;f)$ it holds that $f(V^2)\subseteq A\setminus \{0\}$. Otherwise, $(V;f)$ is not satisfiable.
Note that every $\mathbf{A}$-network $(V;f)$  can be viewed as an $A$-structure $\mathfrak{C}$ on the domain $V$: for all $x,y \in V$ in the domain of $f$ and $a\in A$ the relation $a^\mathfrak{C}(x,y)$ holds if and only if $f(x,y)=a$.


It is well-known that for relation algebras $\bA_1$ and $\bA_2$ the direct product $\bA_1 \times \bA_2$ is also a relation algebra (see, e.g., \cite{HirschHodkinson}). We will see in Lemma~\ref{lem:products} that the direct product of representable relation algebras is also a representable relation algebra. 


\begin{definition}\label{def:unionrep}Let $\bA_1$ and $\bA_2$ be representable relation algebras.
Let $\fB_1$ and $\fB_2$ be  representations of $\bA_1$ and $\bA_2$ with disjoint domains.  Then the \emph{union representation} of the direct product $\bA_1 \times \bA_2$ is  the $(A_1\times A_2)$-structure $\fB_1 \uplus \fB_2$ on the domain $B_1\cup B_2$, 
defined for all $(a_1,a_2)\in A_1\times A_2$ by $$(a_1,a_2)^{\fB_1 \uplus \fB_2}:=a_1^{\fB_1} \cup a_2^{\fB_2}.$$
\end{definition}

The following well-known lemma  establishes a connection between products of relation algebras and union representations 
(see, e.g., Lemma~7 in~\cite{HirschCristiani} or Lemma~3.7 in~\cite{HirschHodkinson}); it states that union representations are indeed representations. Union representations will be the key object in our undecidability proof  for Hirsch's Really Big Complexity Problem.

\begin{restatable}[]{lemma}{lemproducts}\label{lem:products}Let $\bA_1$ and $\bA_2$ be  relation algebras.
Then the following holds:
\begin{enumerate}
    \item If $\fB_1$ and $\fB_2$ are representations of $\bA_1$ and $\bA_2$ with disjoint domains, then $\fB_1 \uplus \fB_2$ is a representation of $\bA_1 \times \bA_2$.
    \item If $\fB$ is a representation of $\bA_1 \times \bA_2$, then there exist representations $\fB_1$ and $\fB_2$ of $\bA_1$ and $\bA_2$ such that $\fB$ is isomorphic to $\fB_1 \uplus \fB_2$. 
\end{enumerate}

\end{restatable}

\begin{proof}The first item can be checked by a straightforward calculation. 
For the second item {note that elements of $\bA_1 \times \bA_2$ are pairs $(a_1,a_2)\in A_1\times A_2$. Since $\fB$ is a representation of $\bA_1 \times \bA_2$, there exists for every $(a_1,a_2)\in A_1\times A_2$ a binary relation $(a_1,a_2)^{\fB} $. 
For better readability, we denote constants of relation algebras by the signature elements $\{ 0,1,\id\}$ (without the superscipt). It will always be clear from the context which algebra is meant. For example, $(0,1)\in A_1\times A_2$ is meant to be the element $(0^{\bA_1},1^{\bA_2})$  of the algebra $\bA_1 \times \bA_2$. }

Consider the sets 
\begin{align*}
B_1 & := \{x\in B\mid (x,x)\in (1,0)^{\fB }\} \\
\text{ and } B_2 & :=\{ x\in B\mid (x,x)\in (0,1)^{\fB }\}.
\end{align*}
We claim that $\{B_1, B_2\}$ forms a partition of $B$. Clearly, $B_1\cup B_2=B$, because $$\{(x,x)\mid x\in B\} =(\id,\id)^{\fB} \subseteq ((1,0)^{\fB }\cup (0,1)^{\fB }).$$ 
By the definition of the relation algebra $\bA_1 \times \bA_2$ it holds that $(0,0)= (1,0)\cap^{\bA_1 \times \bA_2} (0,1)$. Since $\fB$ is a representation of  $\bA_1 \times \bA_2$ we have $\emptyset=(0,0)^{\fB}= (1,0)^{\fB}\cap (0,1)^{\fB}$ 
and $B = (1,1)^{\fB} = (1,0)^{\fB}\cup (0,1)^{\fB}$ 
and 
it follows that $\emptyset=B_1\cap B_2$ and
$B = B_1 \cup B_2$. Hence, $\{B_1, B_2\}$ is a partition of $B$.

Furthermore, we claim that there is no pair $(x,y)\in B_1\times B_2$ in any relation $(a_1,a_2)^{\fB } \subseteq (1,1)^{\fB }$ of $\fB$. So see this, note that for a tuple $(x,y)\in B_1\times B_2$ with $(x,y)\in (1,1)^{\fB }$ it follows from the definition of the relational product $\circ$ that $\{(x,y)\} =(\{(x,x)\}\circ \{(x,y)\} )\circ \{ (y,y)\}$ and therefore $ (x,y)\in ((1,0)^{\fB }\circ (1,1)^{\fB })\circ  (0,1)^{\fB }$ holds. Since this contradicts $\emptyset = (0,0)^{\fB} =((1,0)^{\fB }\circ (1,1)^{\fB })\circ  (0,1)^{\fB }$, there is no pair  $(x,y)\in B_1\times B_2$ in any relation $(a_1,a_2)^{\fB } \subseteq (1,1)^{\fB }$.

Altogether we observe that $B_1$ is the domain of an $(A_1\times \{0\})$-structure $\fB_1$ with $(a,0)^{\fB_1}:=(a,0)^{\fB}$ for every $a\in A_1$. Analogously, the $(\{0\} \times A_2)$-structure $\fB_2$ is defined by $(0,a)^{\fB_2}:=(0,a)^{\fB}$ for every $a\in A_2$. One can check that the mapping $a \mapsto (a,0)^{\fB_1}$ is indeed an isomorphism that witnesses that $\bA_1$ has the representation $\fB_1$. Analogously, we get that   $a \mapsto (0,a)^{\fB}$ witnesses that  $\bA_2$ has the representation $\fB_2$.
\end{proof}

The following result uses Lemma~\ref{lem:products} to obtain reductions between different network satisfaction problems.  A similar statement can be  found in Lemma~7 from~\cite{HirschCristiani}, however
there the assumption on representability of the relation algebras $\bf A$ and $\bf B$  is missing. Note that without this assumption the statement is not longer true. Consider relation algebras $\bf A$ and $\bf B$ such that $\nsp(\bf A)$  is undecidable and $\bf B$ does not have a representation. Then $\bf A \times \bf B$ does also not have a representation (see Lemma~\ref{lem:products}) and hence $\nsp(\bf A \times \bf B)$ is trivial. We observe that the undecidable problem $\nsp(\bf A)$ cannot have a polynomial-time reduction to the trivial problem  $\nsp(\bf A \times \bf B)$.

\begin{lemma}\label{lem:reduce}
Let $\bf A,\bf B \in \rra$ be finite. Then there exists a polynomial-time reduction from $\nsp(\bf A)$ to $\nsp(\bf A \times \bf B)$. 
\end{lemma}
\begin{proof} Consider the following polynomial-time reduction from $\nsp(\bf A)$ to $\nsp(\bf A \times \bf B)$. We map a given $\bA$-network $(V;f)$ to the $(\bf A \times \bf B)$-network $(V;f')$  where $f'$ is defined by $f'(x,y):= (f(x,y),0)$. This reduction can  be computed in polynomial time.

\medskip 
{\bf Claim 1.} If $(V;f)$ is satisfiable then $(V;f')$ is also satisfiable. Let $\fA$ be a representation of $\bA$ in which $(V;f)$ is satisfiable and let $\fB$ be an arbitrary representation of $\bB$. By Lemma~\ref{lem:products}, the structure $\fA \uplus \fB$ is a representation of $\bf A \times \bf B$. Moreover, the definition of union representations (Definition~\ref{def:unionrep}) yields that the $(\bf A \times \bf B)$-network $(V;f')$ is satisfiable in $\fA \uplus \fB$.

\medskip 
{\bf Claim 2.} If $(V;f')$ is satisfiable then $(V;f)$ is satisfiable. Assume that $(V;f')$ is satisfiable in some representation $\fC$ of $\bf A \times \bf B$. By item 2 in Lemma~\ref{lem:products} we get that $\fC$ is isomorphic to $\fA \uplus \fB $, where $\fA$ and $\fB$ are representations of $\bA$ and $\bB$. It follows from  the definition of union representations that $(V;f)$ is satisfiable in the representation $\fA$ of $\bA$.

\medskip This shows the correctness of the polynomial-time reduction from $\nsp(\bf A)$ to $\nsp(\bf A \times \bf B)$ and finishes the proof.
\end{proof}

\subsection{Normal Representations and Constraint Satisfaction Problems}\label{subsec: normal rep csp}\label{subsec:csp}
We consider a subclass of $\rra$ introduced by Hirsch in 1996. For relation algebras $\bf A$ from this class, $\NSP(\bf A)$ corresponds naturally to a constraint satisfaction problem. 
In the following let $\mathbf{A} $ be in $\rra$. We call an $\mathbf{A}$-network $(V;f)$  \emph{closed}~(transitively closed in the work by Hirsch \cite{HirschAlgebraicLogic}) if  
for all $x,y,z \in V$ it holds that 
\begin{itemize}
    \item $f(x,x)\leq \id$,
    \item $f(x,y)=\Breve{a}$ for $a={f(y,x)}$,
    \item $f(x,z) \leq f(x,y) \circ f(y,z)$.
\end{itemize}
It is called \emph{atomic} if the range of $f$ only contains atoms from $\mathbf{A}$.

\begin{definition}[from \cite{Hirsch}]\label{defi:normalrep}
	Let $\mathfrak{B}$ be a representation of $\mathbf{A}$. Then $\mathfrak{B}$ is called 
	\begin{itemize}
		\item \emph{fully universal}, if every atomic closed $\mathbf{A}$-network is satisfiable in $\mathfrak{B}$;
		\item \emph{square}, if $1^\mathfrak{B}=B^2$;
		\item \emph{homogeneous}, if for every isomorphism between finite substructures of $\mathfrak{B}$ there exists an automorphism of $\mathfrak{B}$ that extends this isomorphism;
		\item \emph{normal}, if it is fully universal, square and homogeneous.
	\end{itemize}
\end{definition}

We now investigate the connection between $\nsp(\bA)$ for a finite relation algebra with a normal representation $\fB$ and constraint satisfaction problems.
Let $\tau$ be a finite relational signature and let $\fB$ be a (finite or infinite) $\tau$-structure. Then the \emph{constraint satisfaction problem for $\fB$}, denoted by $\csp(\fB)$, is the computational problem of deciding whether a finite input structure $\fA$ has a homomorphism to $\fB$. The structure  $\fB$ is called the template of $\csp(\fB)$.

Consider the following translation which associates to each  $\mathbf{A}$-network $(V;f)$ an $A$-structure $\fC$ as follows: the set $V$ is the domain of $\fC$ and $(x,y)\in C^2$ is in a relation $a^{\fC}$ if and only if $(x,y)$ is in the domain of $f$ and $f(x,y)=a $ holds. For the other direction let $\fC$ be an $A$-structure with domain $C$ and consider the $\mathbf{A}$-network $(C;f)$ with the following definition: for every $x,y\in C$, if $(x,y)$ does not appear in any relation of $\fC$ we leave $f(x,y)$ undefined, otherwise let $a_1(x,y),\ldots,a_n(x,y) $ be all atomic formulas that hold in $\fC$. We compute in $\mathbf{A}$ the element $a:=a_1 \cap \dots\cap a_n$ and define $f(x,y):=a$. 

	The following theorem  is based on the natural 1-to-1 correspondence between $\mathbf{A}$-networks and $A$-structures; it  subsumes the connection between network satisfaction problems and constraint satisfaction problems.



\begin{proposition}[Proposition 1.3.16 in \cite{Bodirsky-HDR-v8},  see also \cite{Qualitative-Survey,BodirskyRamics}]\label{thm: nsp csp 2}
		Let $\mathbf{A} \in \rra$ be finite. Then the following holds:\begin{enumerate}
			\item there exists a representation $\fB$ of $\mathbf{A}$ such that $\nsp(\mathbf{A})$ and $\csp(\mathfrak{B})$ are the same problem up to the translation between $\mathbf{A}$-networks and $A$-structures.
			\item If $\mathbf{A} $ has a normal representation $\fB$, then the problems $\nsp(\mathbf{A})$ and $\csp(\mathfrak{B})$ are the same up to the translation between $\mathbf{A}$-networks and $A$-structures.
		\end{enumerate}
\end{proposition}
	
Usually, normal representations of relation algebras are infinite relational structures. This means that the transfer from NSPs to CSPs from Proposition~\ref{thm: nsp csp 2} results in CSPs over infinite templates, as in the following example.


\begin{example}\label{exam:point2}Consider the  \emph{point algebra} $\mathbf{P}$. The set of atoms of $\mathbf{P}$ is $P_0=\{\id,<,>\}$. The composition operation $\circ$ on the atoms is given by the multiplication table in Figure~\ref{Tab: PointAl}. The table completely determines the composition operation $\circ$ on all elements of $\mathbf{P}$. 
	\begin{figure}[t]
		
		\begin{center}
			
			\begin{tabular}{|c||c|c|c|}
				\hline 
				$~\circ~$	& $~\id~$ & $~<~$ & $~>~$ \\ 
				\hline \hline
				$\id$	&  $\id$ & $<$ &$>$  \\ 
				\hline 
				$<$	&$<$  & $ <$  & $1 $ \\ 
				\hline 
				$>$	& $>$ & $1$ & $>$ \\ 
				\hline 
			\end{tabular}

		\end{center}
		
		\caption{Multiplication table of the  point algebra $\mathbf{P}$.}
		\label{Tab: PointAl}
	\end{figure}
Note that the structure $\mathfrak{P}:=(\mathbb{Q}; ~\emptyset,<,>,=,\leq, \geq, \not =, \mathbb{Q}^2)$ is the normal representation of $\mathbf{P}$ and therefore $\nsp(\mathbf{P})$  and $\csp(\mathfrak{P})$ are the same problems up to the translation between networks and structures.

	\end{example}

\subsection{The Universal-Algebraic Approach}
We introduce in this section the study of CSPs via the universal-algebraic approach.

\subsubsection{Polymorphisms}
Let $\tau$ be a finite relational signature. 
A \emph{polymorphism} of a $\tau$-structure $\fB$ is a homomorphism $f$ from $\fB^k$ to $\fB$, for
some $k\in \mathbb{N}$ called the \emph{arity of $f$}. We write $\Pol(\fB)$ for the set of all polymorphisms of $\fB$. The set of polymorphisms is closed under composition, i.e., for all $n$-ary $f\in \pol(\fB)$ and $s$-ary $g_1,\ldots, g_n \in \pol(\fB)$  it holds that $f(g_1,\ldots, g_n) \in \pol(\fB)$, where $f(g_1,\ldots, g_n) $  is  a homomorphism from $\fB^s$ to $\fB$ defined as follows
	$$f(g_1,\ldots, g_n) (x_1,\ldots, x_s ):= f( g_1(x_1,\ldots, x_s ) , \ldots, g_n(x_1,\ldots, x_s     )).$$

If $r_1,\dots,r_n \in B^k$ and $f \colon B^n \to B$ an $n$-ary operation, then we write $f(r_1,\ldots,r_n)$ for the $k$-tuple
obtained by applying $f$ component-wise to the tuples $r_1,\ldots,r_n$. We say that $f \colon B^n \to B$ \emph{preserves}  a $k$-ary relation $R\subseteq B^k$  if for all $r_1,\ldots, r_n \in R$ it holds that
	$f(r_1,\ldots,r_n) \in R$.
We want to remark  that the polymorphisms of $\fB$ are precisely those operations that {preserve} all relations from $\fB$.

	 A first-order $\tau$-formula $\varphi(x_1,\ldots,x_n)$ is called \emph{primitive positive (pp)} if it has the  form
$$ \exists x_{n+1},\ldots, x_{m} (\varphi_1\wedge \cdots \wedge \varphi_s)$$
where $\varphi_1, \ldots, \varphi_s$ are atomic $\tau$-formulas, i.e., formulas of the form $R(y_1,\ldots, y_l)$ for $R\in \tau$ and $y_i \in \{x_1,\ldots, x_m\}$, of the form $y=y'$ for $y,y' \in \{x_1,\ldots, x_m\}$, or of the form $\bot$.  
We say that a relation $R$ is \emph{primitively positively definable over $\fA$} if there exists a primitive positive $\tau$-formula  $\varphi(x_1,\ldots,x_n)$  such that $R$ is definable over $\fA$ by $\varphi(x_1,\ldots,x_n)$.
The following result puts together polymorphisms and primitive positive logic.

\begin{proposition}[\hspace{-0.5pt}\cite{Geiger}, \cite{BoKaKoRo}]\label{theo:BKKR}
        Let $\fB$ be a  $\tau$-structure with a finite domain. Then the set of primitive positive definable relations in $\fB$ is exactly the set of relations preserved by $\pol(\fB)$.
\end{proposition}

\subsubsection{Atom Structures}\label{subsec: atomstructure}
In this section we introduce for every finite $\mathbf{A}\in \ra$  an associated finite structure $\fA_0$, called the \emph{atom structure} of $\bf A$ 
(also see~\cite{HirschJacksonK}; the definitions are essentially the same, with one minor technical difference that concerns the signature, mentioned below). If $\bA$ has a fully universal representation, then there exists a polynomial-time reduction from $\Nsp(\bA)$ to the  finite-domain constraint satisfaction problem  $\csp(\fA_0)$ (Proposition~\ref{reductionBtoO}).  Hence, this reduction provides polynomial-time algorithms  to solve NSPs, whenever the CSP of the associated atom structure can be solved in polynomial-time. For a discussion of the atom structure and related objects we recommend Section 4 in \cite{BodKnaeFlexJournal}.

\begin{definition}\label{def: atom str}
The \emph{atom structure of $\bf A \in \ra$} is the finite relational structure ${\fAo}$ with domain $A_0$ and the following relations:

	\begin{itemize}
		\item for every $x \in A$ the unary relation $x^{\fAo}:= \{a\in A_0 \mid   a\leq x \}$,
		\item the binary relation  $E^{\fAo}:= \{ (a_1,a_2)   \in A_0^2  \mid  \Breve{a_1}=a_2 \}$,
		\item the ternary relation $R^{\fAo}:=\{ (a_1,a_2,a_3) \in A_0^3\mid  a_3\leq a_1\circ a_2  \}$.
	\end{itemize}
\end{definition}

Note that $\fAo$ has all subsets of $A_0$ as unary relations\footnote{In contrast to our definition, the atom structure in~\cite{HirschJacksonK} has only one unary relation which contains the set of all atoms below the identity.} and that  the relation $R^{\fAo}$ consists of the allowed triples of  $\bf A \in \rra$. We say that an operation \emph{preserves the allowed triples} if it preserves the relation $R^{\fAo}$.

\begin{proposition}[\hspace{-0.5pt}\cite{BodirskyKnaeAAAI, BodKnaeFlexJournal}]  \label{reductionBtoO}
		Let $\mathfrak{B}$ be a fully universal representation of a finite $\mathbf{A}\in \rra$. Then there is  a polynomial-time reduction from $\csp(\mathfrak{B})$ to $\csp({\fAo})$.
\end{proposition}

\subsubsection{Conservative Clones}

Let $\fB$ be a finite $\tau$-structure. An operation $f \colon B^n \to B$ is called \emph{conservative} if for all $x_1,\ldots, x_n \in B$ it holds that $f(x_1,\ldots,x_n) \in \{x_1,\ldots, x_n\}$. The operation clone $ \Pol(\fB)$ is \emph{conservative} if every $f\in  \Pol(\fB)$ is conservative. We call a relational structure $\fB$ \emph{conservative} if $ \Pol(\fB)$ is conservative.

\begin{remark}\label{rem: Aoconserv} Let $\fAo$ be the atom structure of a finite relation algebra $\bA$. Every $ f\in \pol(\fAo)$ preserves all subsets of $A_0$, and is therefore  conservative. Hence, $\pol(\fAo)$ is conservative.

\end{remark}
This remark justifies our interest in the computational complexity of certain CSPs where the template has conservative polymorphisms. Their complexity can be  studied via universal algebraic methods as we will see in the following. We start with some definitions. An operation $f \colon B^3 \to B$ is called 
\begin{itemize}
\item 
a \emph{majority operation} if $\forall x,y\in B.f(x,x,y) = f(x,y,x) = f(y,x,x) = x$; 
\item a \emph{minority operation} if $\forall x,y\in B. f(x,x,y) = f(x,y,x) = f(y,x,x) = y$.
\end{itemize}
An operation $f \colon B^n \to B$, for $n \geq 2$, 
is called 
\begin{itemize}
\item a 
\emph{cyclic operation} if 
	$\forall x_1,\ldots, x_n \in B. f(x_1,\ldots,x_n)= f(x_n,x_1, \ldots, x_{n-1})$; 
 \item a \emph{weak near-unanimity operation} if 
	$$\forall x,y\in B. f(x,\ldots,x,y)= f(x,\ldots,x,y,x)=\ldots =f(y,x,\ldots,x);$$ 
 \item a \emph{Siggers operation} if $n = 6$ and 
		$\forall x,y\in B. f(x,x,y,y,z,z)= f(y,z, x,z,x,y ).$
\end{itemize}

The following terminology was introduced by Bulatov  and has proven to be extremely powerful, especially in the context of conservative clones.

\begin{definition}[\hspace{-0.5pt}\cite{Conservative,BulatovConservative}]\label{defi:edgescolor}
	A pair $(a,b) \in B^2$ is called a \emph{semilattice edge} if there exists $f \in \Pol(\fB)$ of arity two such that $f(a,b) = b = f(b,a) = f(b,b)$ and $f(a,a) = a$. We say that a two-element set $\{a,b\}\subseteq B$ \emph{has a semilattice edge} if $(a,b)$ or $(b,a)$ is a semilattice edge.
	
	  A two-element subset $\{a,b\}$ of $B$ is called a \emph{majority edge} if neither $(a,b)$ nor $(b,a)$ is a semilattice edge and there exists an $f \in \Pol(\fB)$ of arity three whose restriction to $\{a,b\}$ is a majority operation. 
	  
	A two-element subset $\{a,b\}$ of $B$ is called an \emph{affine edge} if it is not a majority edge, if neither $(a,b)$ nor $(b,a)$ is a semilattice edge,  and there exists an $f \in \Pol(\fB)$ of arity three whose restriction to $\{a,b\}$ is a minority operation. 
	
\end{definition}

{
If 
$S \subseteq B$ 
and $(a,b) \in S^2$ is a semilattice edge then we say that \emph{$(a,b)$ is a semilattice edge on $S$}.
Similarly, if 
$\{a,b\} \subseteq S$ is a majority edge (affine edge) then we say that $\{a,b\}$ is a \emph{majority edge on $S$} (\emph{affine edge on $S$}). }

According to  Definition~\ref{defi:edgescolor}, an ``edge type'' of a concrete set $\{a,b\}\subseteq B$ is witnessed by a certain operation. For another set $\{c,d\}\subseteq B$ this could a priori be a different operation (even if the two sets have the same edge type). However, Bulatov obtained ``uniform witness operations'' by the following proposition.

\begin{proposition}[Proposition 3.1 in~\cite{BulatovConservative}]\label{prop:bulatov color functions}
	Let $\fB$ be a finite structure. 
	Then there are a binary operation $v \in \Pol(\fB)$ and ternary operations $g,h \in \Pol(\fB)$ such that for every 
	two-element subset $C$
	of $B$ we have that
	\begin{itemize}
		\item $v|_{C}$ is a semilattice operation whenever $C$ has a semilattice edge, and $v|_{C}(x,y)=x$ otherwise;
		\item $g|_{C}$ is a majority operation if $C$ is  a majority edge, $g|_{C}(x,y,z)=x $ if $C$ is affine and $g|_{C}(x,y,z)=v|_{C}( v|_{C}(x,y),z)$ if $C$ has a semilattice edge;
		\item $h|_{C}$ is a minority operation if $C$ is  an affine edge, $h|_{C}(x,y,z)=x $ if $C$ is majority and $h|_{C}(x,y,z)=v|_{C}( v|_{C}(x,y),z)$ if $C$ has a semilattice edge.
		
	\end{itemize}
\end{proposition}

The main result about conservative finite structures and their CSPs is the following dichotomy, first proved by Bulatov, 14 years before the proof of the Feder-Vardi conjecture.

\begin{theorem}[\hspace{1sp}\cite{Conservative}; see also~\cite{Barto-Conservative,BulatovConservative,Bulatov-Conservative-Revisited}]\label{2elemtsubalgebrassiggers}
	Let $\fB$ be a finite structure with a finite relational signature such that $\Pol(\fB)$ is conservative. Then precisely one of the following holds:
	\begin{enumerate}
			\item $\Pol(\fB)$ contains a Siggers operation; 
   in this case, $\csp(\fB)$ is in P. 
		\item There exist distinct $a,b \in B$ such that for every $f \in \Pol(\fB)^{(n)}$ the restriction
		of $f$ to $\{a,b\}^n$ is a projection. In this case, $\csp(\fB)$ is NP-complete. 
	\end{enumerate}
\end{theorem}
This means that $\Pol(\fB)$ contains a Siggers operation if and only if  for all two elements $a,b \in B$ the set $\{a,b\} $ is a majority edge, an affine edge, or 
there is a semilattice edge on $\{a,b\}$ (see~\cite{BulatovConservative}).

\subsection{The Consistency Procedure}
We present in the following the $k$-consistency procedure. It was introduced in \cite{AtseriasBulatovDalmau1} for finite structures and extended to infinite structures in several equivalent ways, for example in terms of  Datalog programs, existential pebble games, and finite variable logics \cite{BodDalJournal}. Also see~\cite{Collapses} for recent results about the power of $k$-consistency for infinite-domain CSPs.

Let $\tau$ be a finite relational signature and let $k, l\in \mathbb{N}$ with $k<l$ and let $\fB$ be a fixed  $\tau$-structures with finitely many orbits of $l$-tuples, i.e., finitely many sets of the form $\{(\alpha(t_1),\dots,\alpha(t_l)) \mid \alpha \in \Aut(\fB) \}$ for some $t_1,\dots,t_l \in B$.
We define $\fB'$ to be the  expansion of $\fB$ by all orbits of $n$-tuples for every $n\leq l$, i.e., we add new relation symbols to the signature $\tau$ that denote these orbits. We write $\tau'$ for the extended signature of $\fB'$. 
 Let $\fA$ be an arbitrary finite $\tau$-structure. 
 A \emph{partial $l$-decoration of $\fA$} is a set $g$ of atomic $\tau'$-formulas such that
 \begin{enumerate}
     \item the variables of the formulas from $g$ are a subset of $A$ and denoted by $\var(g)$,
     \item $|\var(g)|\leq l$,
     \item the $\tau$-formulas in $g$ hold in $\fA$, where variables are interpreted as domain elements of the relational structure $\fA$,
     \item the conjunction over all formulas in $g$ is satisfiable in $\fB'$.
 \end{enumerate}

 A partial $l$-decoration $g$ of $\fA$ is called \emph{maximal} if there exists no partial $l$-decoration $h$ of $\fA$ with $\var(g)=\var(h)$ such that $g\subsetneq h$.
 We denote  the set of maximal partial $l$-decorations of $\fA$ by $\mathcal{R}^{l}_{\fA}$. 
 Note that a fixed finite set of at most $l$ variables, there are only finitely many partial $l$-decorations of $\fA$, because $\fB$ has by assumption finitely many orbits of $l$-tuples. Since this set is constant and can be precomputed, the set $\mathcal{R}^l_{\fA}$ 
can be computed efficiently.
  Then the \emph{$(k,l)$-consistency procedure for $\fB$} is the following algorithm.

\begin{algorithm}\label{alg:main}
\DontPrintSemicolon{}
\SetKwInOut{Input}{Input}

\SetKwHangingKw{Compute}{compute}

\Input{A finite $\tau$-structure $\fA$. }

\Compute {$\mathcal{H} := \mathcal{R}^{l}_{\fA}$.} 

\Repeat{$\mathcal{H}$ does not change}{ For every $f\in \mathcal{H}$ with $\var(f)\leq k$ and every $U\subseteq A$ with $|U|\leq l-k$, if there does not exist $g\in \mathcal{H}$ with $f\subseteq g$ and $U\subseteq\dom(g)$, then remove $f$ from $\mathcal{H}$.}

\uIf{$\mathcal{H}$ is empty}{\Return Reject.}
    \uElse{\Return Accept.}

\caption{$(k,l)$-consistency procedure for $\fB$} 
\end{algorithm}


Since $\mathcal{R}^{l}_{\fA}$ is of  polynomial size (in the size of $A$) and the  $(k,l)$-consistency procedure removes in step 3.\ at least one element from $\mathcal{R}^{l}_{\fA}$  the algorithm has a polynomial run time.
The $(k,k+1)$-consistency procedure is also called \emph{$k$-consistency procedure}.
The $(2,3)$-consistency procedure is called \emph{path consistency procedure}.\footnote{Some authors also call it the  \emph{strong path consistency algorithm}, because some forms of the  definition of the path consistency procedure are only equivalent to our definition of the path consistency procedure if $\fB$ has a transitive automorphism group.}

\begin{definition}
Let  $\fB$ be a relation $\tau$-structure as defined before.
   Then the \emph{$(k,l)$-consistency procedure for $\fB$ solves $\csp(\fB)$} if the satisfiable instances of $\csp(\fB)$ are precisely the accepted instances of the $(k,l)$-consistency procedure.
\end{definition}



	\begin{remark}\label{rem:NSP consistency}
		Let $\bA$ be a relation algebra with a normal representation $\fB$.
		We will in the following say that the $k$-consistency procedure solves $\Nsp(\bA)$  if it  solves $\Csp(\fB)$. This definition is justified by the correspondence of NSPs and CSPs from Theorem~\ref{thm: nsp csp 2}.
		\end{remark}

\begin{theorem}[\hspace{-0.5pt}\cite{Maltsev-Cond}]\label{theo:datalog finite} 
Let $\fB$ be a finite $\tau$-structure. Then the following statements are equivalent:
	\begin{enumerate}
		\item There exists $k\in \mathbb{N}$ such that the $k$-consistency procedure solves $\csp(\fB)$.
		\item $\fB$ has a $3$-ary weak near-unanimity polymorphism $f$ and a $4$-ary weak near-unanimity polymorphism $g$ such that:  $	\forall x,y,z \in B.~ f(y,x,x)=g(y,x,x,x).$
		
	\end{enumerate}
	
\end{theorem}

Let $\fAo$ be the atom structure of a relation algebra $\bA$ with a normal representation $\fB$.
We finish this section by connecting the  solvability of $\Csp(\fAo)$ by $k$-consistency  (or its characterization in terms of polymorphims from the previous proposition) 
with the solvability of $\Csp(\fB)$  by $k$-consistency. By Remark~\ref{rem:NSP consistency} this gives a criterion for the solvability of $\NSP(\bA)$ by the $k$-consistency procedure.

The following theorem is from  \cite{Collapses}  building on ideas from \cite{BodMot-Unary}. We present it here in a specific formulation that already incorporates a correspondence between polymorphisms of the atom structure and canonical operations. For more details see \cite{BodirskyKnaeAAAI, BodKnaeFlexJournal}.

\begin{theorem}[\hspace{-0.5pt}\cite{Collapses}]\label{theo: datalog reduction}
Let $\fB$ be a normal representation of a finite relation algebra $\bA$ and $\fAo$ the atom structure $\bA$. If  $\Pol(\fAo)$ contains a $3$-ary weak near-unanimity polymorphism $f$ and a $4$-ary weak near-unanimity polymorphism $g$ such that $$	\forall x,y,z \in B.~ f(y,x,x)=g(y,x,x,x),$$
	then $\NSP(\bA)$ is solved by the $(4,6)$-consistency algorithm.
\end{theorem}


\section{The Undecidability of RBCP, CON, and PC}
In order to view the really big complexity problem (RBCP,~\cite{Hirsch}) as a decision problem, we need the following definitions. 
Let $\fra$ be the set of all relation algebras $\bf A$ whose domain is 
$\mathcal{P}(\{0,\ldots,n-1\})$, the set of all subsets of the first $n$ natural numbers, for some $n \in {\mathbb N}$. 

\begin{definition}[PN, CON, PC]
We define the following subsets of $\fra$: 
\begin{itemize}
    \item 
$\rbcp$ denotes the set of all ${\bf A}$ such that $\nsp({\bf A})$ is in P.
\item $\rbcp^c$ denotes $\fra \setminus \rbcp$.
\item $\kC$ denotes the set of all ${\bf A}$ such that $\nsp({\bf A})$ is solved by  $k$-consistency for some $k\in \mathbb{N}$.
\item $\pc$ denotes the set of all ${\bf A}$ such that $\nsp({\bf A})$ is solved by path consistency. 
\end{itemize}
\end{definition}

The following theorem is our first  result. Note that the undecidability of $\rbcp$ can be seen as a negative answer to Hirsch's Really Big Complexity Problem \cite{Hirsch}.
\begin{theorem}\label{thm:RBCP-undecidable}
{$\rbcp$ is undecidable, $\kC$ is undecidable, and $\pc$ is undecidable. }
\end{theorem}

In our undecidability proofs we reduce from the following well-known undecidable problem for relation algebras~\cite{HirschHodkinsonRepresentability}.

\begin{definition}[Rep]
Let $\rep$ be the computational problem of deciding for a given $\bf A \in \fra$ whether $\bf A$ has a representation. 
\end{definition}

In our proof we also use the fact that there exists a ${\bf U} \in \fra$ such that $\nsp({\bf U})$ is undecidable~\cite{Hirsch-Undecidable}. Note that then ${\bf U} \in \rep$, because the network satisfaction problem for non-representable relation algebras is trivial and therefore decidable.

\begin{proof}[Proof of Theorem~\ref{thm:RBCP-undecidable}]
We reduce the problem $\rep$ to $\rbcp^c$. Consider the following reduction $f\colon \fra \to \fra$.
For a given $\bf A \in \fra$, we define 
$f({\bf A}) := {\bf A}  \times {\bf U}$.

\medskip 
{\bf Claim 1.} If ${\bf A} \in \rep$ then $f({\bf A}) \in \rbcp^c$. 
If ${\bf A}$ is representable, then 
${\bf A} \times {\bf U}$ is representable by the first part of Lemma~\ref{lem:products}. 
Then there is a polynomial-time reduction from $\nsp({\bf U})$ to 
 $\nsp({\bf A} \times {\bf U})$ by Lemma~\ref{lem:reduce}. 
 This shows that $\nsp({\bf A} \times {\bf U})$ 
is undecidable, and hence $f({\bf A})$ is in $\rbcp^c$. 

\medskip 
{\bf Claim 2.} 
If ${\bf A} \in \fra \setminus \rep$ then $f({\bf A}) \in \rbcp$. 
If {\bf A} is not representable,
then ${\bf A} \times {\bf U}$ is not  representable by the second part of Lemma~\ref{lem:products}, and hence
$\nsp({\bf A} \times {\bf  U})$ is trivial and in P, and therefore in $\rbcp$. 

Clearly, $f$ is computable (even in polynomial time). Since $\rep$ is undecidable~\cite{HirschHodkinsonRepresentability}, this shows that $\rbcp^c$, and hence $\rbcp$, is undecidable as well. The proof for $\kC$ and $\pc$ is analogous; all we need is the fact that $\nsp({\bf U}) \notin 
\kC$ and $\nsp({\bf U}) \notin 
\pc$.
\end{proof}

\section{Tractability via \texorpdfstring{$k$}{}-Consistency}\label{sec: LocalCons}
We provide in this section a criterion that ensures solvability of NSPs by the $k$-consistency procedure (Theorem~\ref{theo:Datatractability}). 
A relation algebra $\mathbf{A}$ is called \emph{symmetric} if all its elements are symmetric, i.e., $\Breve{a}=a$ for every $a\in A$. Note that the relation $E$ of the atom structure is in this case simply the equality relation. 
We will see in the following that the assumption on $\mathbf{A}$ to be symmetric will simplify the atom structure $A_0$ of $\bA$, which has some advantages in the upcoming arguments.

\begin{definition}\label{defi:sigersbeh}
    Let $\bf A$ be a finite symmetric relation algebra with set of atoms  $A_0$. We say that $\bf A$ admits a \emph{Siggers behavior} if there exists an operation 
	$s\colon A_0^6\rightarrow A_0$ such that
	\begin{enumerate}
		\item $s$ preserves the allowed triples of $\mathbf{A}$,
		\item 	$\forall x_1,\ldots,x_6\in A_0.~s(x_1,\ldots,x_6)\in \{x_1,\ldots ,x_6\}$,
		\item $s$  satisfies the {Siggers identity}: 
		$\forall x,y,z \in A_0.~ s(x,x,y,y,z,z)=s(y,z,x,z,x,y).$
  \end{enumerate}
\end{definition}

\begin{remark}
For readers familiar with the theory of infinite-domain CSPs, we mention that if a relation algebra 
$\bf A$ has a normal 
representation $\fB$, then $\bf A$
admits a Siggers behavior if and only if $\fB$ has a pseudo-Siggers polymorphism which is canonical with respect to $\Aut(\fB)$; see~\cite{BodMot-Unary}. 
\end{remark}

We say that a finite symmetric relation algebra $\bf A$ \emph{has all $1$-cycles} if for every $a\in A_0$ the triple $(a,a,a)$ is allowed. Details on the notion of cycles from the relation algebra perspective can be found in \cite{Maddux2006-dp}. The relevance of the existence of 1-cycles for constraint satisfaction comes from the following observation.

\begin{lemma}\label{lem:injective implies cycle}Let $\bA$ be a finite symmetric relation algebra with a representation $\fB$ that has a binary injective polymorphism. Then $\bA$ has all $1$-cycles.
\end{lemma}

\begin{proof} Let $f$ be a {binary} injective polymorphism of $\fB$.
Arbitrarily choose $a\in A_0$. If $a \leq \id$, then $(a,a,a)$ is clearly an allowed triple, so suppose that this is not the case, i.e., $a \cap \id = 0$. 
	Consider $x_1,x_2, y_1,y_2\in B$ such that $a^{\fB}(x_1,x_2) $ and $a^{\fB}(y_1,y_2)$. 
	Then $r_0 := (f(x_1,x_2),f(x_2,y_2)) \in a^{\fB}$ by definition. Since $f$ preserves $a \cup \id$, we have that
	\begin{align*}
	& r_1 := (f(x_1,y_1),f(x_1,y_2)) \in (a \cup \id)^{\bB} \\
	\text{and } \quad 
	& r_2 := (f(x_1,y_1),f(x_2,y_1)) \in (a \cup \id)^{\bB}.
	\end{align*} 
	We now consider $t := f((x_1,x_1,x_2),(y_1,y_2,y_2)) = (f(x_1,y_1),f(x_1,y_2),f(x_2,y_2))$. 
	Since 
	$a \cap \id = 0$, 
	the injectivity of $f$ implies that 
	both $r_1$ and $r_2$ lie in $a$ as well.  
	Hence, $t$ witnesses that $(a,a,a)$ is an allowed triple. 
\end{proof}

\begin{theorem}\label{theo:Datatractability} 
Let $\bf A$ be a finite symmetric relation algebra  with a normal representation $\fB$. 
Suppose that the following holds:
\begin{enumerate}
    \item  $\bf A$ has all $1$-cycles.
    \item $\bf A$ admits a Siggers behavior.
\end{enumerate}
	Then the $\nsp(\mathbf{A})$ can be solved by the $(4,6)$-consistency procedure.
\end{theorem}

We will outline the proof of Theorem~\ref{theo:Datatractability} and cite some results from the literature that we will use.
Assume that $\bA$ is a  finite symmetric relation algebra that satisfies the assumptions of Theorem~\ref{theo:Datatractability}.  Since $\bA$ admits a Siggers behavior there exists an operation $s\colon A_0^6\rightarrow A_0$ that is by 1.\ and 2.\  in Definition~\ref{defi:sigersbeh} a polymorphism of the atom structure $\fAo$ (see Paragraph~\ref{subsec: atomstructure}). 
By Remark~\ref{rem: Aoconserv},  $\Pol(\fAo)$ is a conservative operation clone.
Recall the notion of  semilattice, majority, and affine edges for conservative clones (cf.~Definition~\ref{defi:edgescolor}). 
Since $s$ is by 3.\  in Definition~\ref{defi:sigersbeh} a Siggers operation, Theorem~\ref{2elemtsubalgebrassiggers} implies that every edge in $\fAo$ is  semilattice, majority, or affine. 

Our goal is to show that there are no affine edges in $\fAo$, since this implies {that there exists $k\in \mathbb{N}$ such that
 $\Csp(\fAo)$ can be solved by $k$-consistency~\cite{BulatovConservative}. We present this fact here via the  characterization of $(k,l)$-consistency  in terms of weak near-unanimity polymorphisms from Theorem~\ref{theo:datalog finite}.

\begin{proposition}[cf.\  Corollary 3.2 in~\cite{Kazda2015}]\label{prop:no affine implies bounded w} 
	Let $\fAo$ be a finite conservative relational structure with a Siggers polymorphism and no affine edge. Then	$\fAo$ has a $3$-ary weak near-unanimity polymorphism $f$ and a $4$-ary weak near-unanimity polymorphism $g$ such that $$	\forall x,y,z \in B.~ f(y,x,x)=g(y,x,x,x).$$
\end{proposition}

Note that the existence of the weak near-unanimity polymorphisms  from  Proposition~\ref{prop:no affine implies bounded w} would finish the proof of Theorem~\ref{theo:Datatractability}, because Theorem~\ref{theo: datalog reduction} implies that in this case $\nsp(\mathbf{A})$ can be solved by the $(4,6)$-consistency procedure.  
We therefore want to prove that there are no affine edges in $\fAo$. We start in Section~\ref{sec:atomstructure datalgo} by analyzing the different types of edges in the atom structure $\fAo$ and obtain results about their appearance. 

Fortunately, there is the following result by Alexandr Kazda about binary}  structures with a conservative polymophism clone. A \emph{binary structure} is a structure  where all relations have arity at most two.

\begin{theorem}[Theorem 4.5 in~\cite{Kazda2015}]
	\label{thm:kazda} 
	If $\fA$ is a finite  binary conservative relational structure  with a 
	Siggers polymorphism, 
	 then $\fA$ has no affine edges. 
\end{theorem}
 
Notice that we cannot simply apply this theorem to the atom structure $\fAo$, since the maximal arity of its relations is three. We circumvent this obstacle by defining for $\fAo$ a closely related binary structure $\fAob$, which we call the ``binarisation of $\fAo$''. In Section~\ref{sec:bina} we give the formal definition of $\fAob$ and investigate how $\Pol(\fAo)$ and $\Pol(\fAob)$ relate to each other.
It follows from these observations that $\fAob$ does not have an affine edge. In other words, it only has semilattice and majority edges. The crucial step in our proof is to transfer a witness of this fact to $\fAo$ and conclude that also $\fAo$  has no affine edge. This is done in Section~\ref{sec: no affine}.

\subsection{The Atom Structure}\label{sec:atomstructure datalgo}
\label{sect:atomstruct}
For the sake of notation, 
we make some global assumptions for Sections~\ref{sec:atomstructure datalgo}--\ref{sec: no affine}. Let $\bA$ be a finite relation algebra that satisfies the assumptions from Theorem~\ref{theo:Datatractability}. 
We denote by $\fAo$ the atom structure of $\bA$ (Definition~\ref{def: atom str}). Since $\bA$ is a symmetric relation algebra, the relation $R^{\fA_0}$ is totally symmetric. Furthermore, we can drop the binary relation $E^{\fA_0}$, since it consists only of loops and does not change the set of polymorphisms.  
Let $s \in \Pol(\fAo) $ be the Siggers operation that exists by the assumptions in Theorem~\ref{theo:Datatractability}. 
This implies by Theorem~\ref{2elemtsubalgebrassiggers} for every $a,b\in A_0$ that the set $\{a,b\} $ is a majority edge or an affine edge, or that 
there is a semilattice edge on $\{a,b\}$. 
The different types of edges are witnessed by certain operations that we get from Proposition~\ref{prop:bulatov color functions}: there exist a binary operation $f \in \Pol(\fAo)$ and ternary operations $g, h\in \Pol(\fAo)$ such that for every two element subset $C$ of $A_0$,

\begin{itemize}
	\item $f|_{C}$ is a semilattice operation whenever $C$ has a semilattice edge, and $f|_{C}(x,y)=x$ otherwise;
	\item $g|_{C}$ is a majority operation if $C$ is  a majority edge, $g|_{C}(x,y,z)=x $ if $C$ is affine and $g|_{C}(x,y,z)=f|_{C}( f|_{C}(x,y),z)$ if $C$ has a semilattice edge;
	\item $h|_{C}$ is a minority operation if $C$ is  an affine edge, $h|_{C}(x,y,z)=x $ if $C$ is majority and $h|_{B}(x,y,z)=f|_{C}( f|_{C}(x,y),z)$ if $C$ has a semilattice edge.
\end{itemize}

We will fix these operations and introduce the following terminology. 
A tuple $(a,b)\in A_0$ is called \emph{$f$-sl} if $f(a,b)=b=f(b,a)$ holds.
Next, we prove several important properties 
of the relation $R$: that it must contain certain triples (Lemma~\ref{lem: aaa in R}),
that it must not contain certain other triples (Lemma~\ref{lem:taylor}), 
and that it is affected by the presence of semilattice edges in $\bA_0$ (Lemma~\ref{lem:aabforb} and Lemma~\ref{lem:no-cycle}).

\begin{lemma}\label{lem: aaa in R}The relation $R$ of the atom structure $\fAo$ has the following properties:
\begin{itemize}
    \item for all $a\in A_0$ we have $(a,a,a) \in R$.
	\item for all $a, b \in A_0$ we have $(a,a,b) \in R$ or $(a,b,b)\in R$; 
	
\end{itemize}

\end{lemma}

\begin{proof} The first item follows from the assumption that $\bA$ has all 1-cycles. 

For the second item observe that $\{a, \id\}$ cannot be a majority edge. Otherwise, $$g((a,a,\id),(\id,a,a),(\id,\id,\id))=(\id,a,\id)\in R$$ is a contradiction to the properties of $\id$. Furthermore, $(a,\id)$ cannot be $f$-sl, since $$f((a,a,\id),(\id,a,a))=(\id,a,\id)\in R.$$ This is again a contradiction. 
Since these observations also hold for $b$ instead of $a$ we have the following case distinction.

\begin{enumerate}
    \item $(\id, a)$ is $f$-sl and $(\id, b)$ is $f$-sl. It follows that $f((a,a,\id),(\id,b,b))\in \{(a,a,b), (a,b,b)\}$.  Since $f$ preserves $R$, $(a,a,\id)\in R$, and $(\id,b,b) \in R$ we get that $f((a,a,\id),(\id,b,b))\in R$. This implies that $(a,a,b) \in R$ or $(a,b,b)\in R$.
    \item $(\id, a)$ is $f$-sl  and $\{b, \id\}$ is affine.
    By the definition of $f$ we get $f((b,b,\id),(\id,a,a))\in \{(b,a,a), (b,b,a)\}$. By the same argument as in Case 1 we get that $(a,a,b) \in R$ or $(a,b,b)\in R$.
     \item $(\id, b)$ is $f$-sl  and $\{a, \id\}$ is affine. This case is analogous to Case 2.
      \item $\{a, \id\}$ is affine  and $\{b, \id\}$ is affine. Observe that  $$g((a,\id,a),(\id,b,b),(\id,\id,\id))\in \{(a,b,a), (a,b,b)\},$$ since $g(a,b,\id) \in \{a,b,\id\}$ and the triple $(a,b,\id)$ is forbidden. As in the cases before it follows that $(a,a,b) \in R$ or $(a,b,b)\in R$.
\end{enumerate}
This concludes the proof of the second item. 
\end{proof}


%
%

\begin{figure}[t]
	\centering
	
\begin{tabular}{cc}
\begin{tikzpicture}[scale=1]
	
	\node (a) at (-0.7,0) [circle,fill,inner sep=1pt, outer sep=2pt] {};
	\node (b) at (-0.7,1)[circle,fill,inner sep=1pt , outer sep=2pt]{};
	\node (c) at (-0.7,2) [circle,fill,inner sep=1pt, outer sep=2pt] {};
	
	\node () at (-1.1,0)[]{$a$};
	
	\node () at (-1.1,1)[]{$b$};
	\node () at (-1.1,2)[]{$c$};

	
	\node (al) at (-0.85,-0.1){} ;
	\node (a'l) at (0.38,-0.14){} ;
	\node (bl) at (-0.85,1){};
	\node (bl2) at (-0.86,0.9){};
	\node (cl) at (-0.85,2.1) {};
	
	\node (ar) at (-0.55,-0.1) {};
	\node (a'r) at (0.65,0.1) {};
	\node (br) at (-0.55,1){};
	\node (br2) at (-0.5,1){};
	\node (cr) at (-0.55,2.1){} ;

	
	\draw	[draw=red, 
	fill=red,  fill opacity=0.2] plot [smooth cycle] coordinates { (al) (ar)(br)(cr)(cl)(bl) };
	
		\node (al) at (1,1){{\LARGE $\Rightarrow$}} ;


\end{tikzpicture}

	&  
	\begin{tikzpicture}[scale=1]
			\node (al) at (-1.5,1){} ;
		\node (a) at (-0.7,0) [circle,fill,inner sep=1pt, outer sep=2pt] {};
		\node (b) at (-0.7,1)[circle,fill,inner sep=1pt , outer sep=2pt]{};
		\node (c) at (-0.7,2) [circle,fill,inner sep=1pt, outer sep=2pt] {};
		
		\node () at (-1.1,0)[]{$a$};
		
		\node () at (-1.1,1)[]{$b$};
		\node () at (-1.1,2)[]{$c$};

		
		\node (al) at (-0.85,-0.1){} ;
		\node (a'l) at (0.38,-0.14){} ;
		\node (bl) at (-0.85,1){};
		\node (bl2) at (-0.86,0.9){};
		\node (cl) at (-0.85,2.1) {};
		
		\node (ar) at (-0.55,-0.1) {};
		\node (a'r) at (0.65,0.1) {};
		\node (br) at (-0.55,1){};
		\node (br2) at (-0.5,1){};
		\node (cr) at (-0.55,2.1){} ;

		
		\draw	[draw=red, 
		fill=red,  fill opacity=0.2] plot [smooth cycle] coordinates { (al) (ar)(br)(cr)(cl)(bl) };
		

		\draw[->, line width = 1pt](a) to (b);

	\end{tikzpicture}

	\\
\end{tabular}
	
	\caption[Illustration for the proof od Lemma~\ref{lem:taylor}.]{The statement of Lemma~\ref{lem:taylor}.
		The red shape means $(a,b,c)\notin R$, the black arrow means $(a,a,b)\notin R$.
	}\label{fig:taylor}
\end{figure}

\begin{lemma}\label{lem:taylor}
	Let $a,b,c \in A_0$ {be such that $(a,b,c)\notin R$ and $|\{a,b,c\}|=3$.}
	Then there are $x,y \in \{a,b,c\}$ such that $(x,x,y) \notin R$. 
\end{lemma}
\begin{proof}
	We first suppose that there is a semilattice edge on $\{a,b,c\}$. 
	Without loss of generality we assume that $(a,b)$ is  $f$-sl.
	If $f(c,a)=c$ 
	then $(a,a,c) \notin R$ 
	or $(b,a,a) \notin R$ because otherwise
	$$f((a,a,c),(b,a,a))=(b,a,c) \in R$$ contradicting our assumption. 
	If  $f(c,a)=a$ 
	then $(b,c,c) \notin R$ or $(a,a,c) \notin R$ because otherwise 
	$$f((b,c,c),(a,a,c))=(b,a,c) \in R$$ which is again a contradiction. 
	Hence, in all the cases we found $x,y \in \{a,b,c\}$ such that $(x,x,y) \notin R$ and are done. 
	In the following we therefore assume that there is no semilattice edge on $\{a,b,c\}$. 
	
	Next we suppose that there is an affine edge on $\{a,b,c\}$. 
	Without loss of generality we  assume that $\{a,b\}$ is an affine edge. Since there are no semilattice edges on $\{a,b,c\}$ we distinguish the following two cases:
	\begin{enumerate}
		\item $\{a,c\}$ is an affine edge. In this case 
		$(c,a,a) \notin R$ or $(a,b,a) \notin R$ because otherwise
		$$h((c,a,a), (a,a,a),(a,b,a))= (c,b,a) \in R.$$
		\item $\{a,c\}$ is a majority edge. In this case 
		$(a,a,c) \notin R$ or $(a,b,a) \notin R$ or $(b,b,c) \notin R$, because otherwise 
		$$h((a,a,c), (a,b,a),(b,b,c))= (b,a,c) \in R.$$
	\end{enumerate}
	In both cases we again found $x,y \in \{a,b,c\}$ such that $(x,x,y) \notin R$ and are done. We therefore suppose in the following that there are no affine edges on $\{a,b,c\}$. 
	Hence, all edges on $\{a,b,c\}$ are majority edges. 
	Then $(a,a,c) \notin R$ or $(a,b,a) \notin R$ or $(b,b,c) \notin R$ because otherwise
	$$g((a,a,c), (a,b,a),(b,b,c)) =(a,b,c) \in R.$$ 
	Thus, also in this case we found
	$x,y \in \{a,b,c\}$ such that $(x,x,y) \notin R$. 
\end{proof}
%
%

The next lemma states that the edge type on $\{a,b\}$ is predetermined whenever a triple $(a,a,b)$ is not in $R$.

\begin{lemma}\label{lem:aabforb}
	Let $a, b\in A_0$ be such that $(a,a,b) \notin R$. Then $(a,b)$ is a semilattice edge {in $\fAo$}
	but $(b,a)$ is not. 
\end{lemma}

\begin{proof}
	By Lemma~\ref{lem: aaa in R} we know that $(a,b,b) \in R$, $(a,a,a) \in R$, and $(b,b,b) \in R$.
	Assume for contradiction that $\{a,b\}$ is a majority edge. 
	Then 
	$$
	g((a,a,a),(a,b,b), (b,b,a) )= (a,b,a) $$ which contradicts the fact that $g$ preserves $R$.
	Assume next that $\{a,b\}$ is an affine edge. 
	Then $$
	h((a,b,b),(b,a,b), (b,b,b) )= (a,a,b) $$ which again contradicts the fact that $h$ preserves $R$.
	Finally, 
	if $(b,a)$ is a semilattice edge then 
	$$f((a,b,b),(b,a,b))=(a,a,b)$$
	which contradicts the assumption that $f$ preserves $R$.  
	If follows that $(a,b)$ is the only semilattice edge on $\{a,b\}$ and therefore $f(a,b)=b=f(b,a)$ holds.
\end{proof}

\begin{figure}[t]
	\centering
	
	\begin{tabular}{cc}

	\begin{tikzpicture}[scale=1]

		\node (a) at (-0.7,0) [circle,fill,inner sep=1pt, outer sep=2pt] {};
		\node (b) at (-0.7,1)[circle,fill,inner sep=1pt , outer sep=2pt]{};
		\node (c) at (-0.7,2) [circle,fill,inner sep=1pt, outer sep=2pt] {};
		\node (a') at (0.5,0) [circle,fill,inner sep=1pt, outer sep=2pt] {};

		\node () at (-1.1,0)[]{$a$};
		\node () at (0.9,0.05)[]{$a'$};
		
		\node () at (-1.1,1)[]{$b$};
		\node () at (-1.1,2)[]{$c$};

		
			\node (al) at (-0.85,-0.1){} ;
			\node (a'l) at (0.38,-0.14){} ;
		\node (bl) at (-0.85,1){};
			\node (bl2) at (-0.86,0.9){};
		\node (cl) at (-0.85,2.1) {};
		
		\node (ar) at (-0.55,-0.1) {};
		\node (a'r) at (0.65,0.1) {};
		\node (br) at (-0.55,1){};
		\node (br2) at (-0.5,1){};
		\node (cr) at (-0.55,2.1){} ;

		\draw	[draw=blue, 	fill=blue,  fill opacity=0.2] plot [smooth cycle] coordinates {  (a'l)(a'r)(br2)(cr) (cl)(bl2)};
		
		\draw	[draw=red, 
		fill=red,  fill opacity=0.2] plot [smooth cycle] coordinates { (al) (ar)(br)(cr)(cl)(bl) };
		
%


			\draw[->, line width = 1pt](a) to (b);

			\node () at (2,1){{\LARGE $\Rightarrow$}} ;
		
	\end{tikzpicture}

&

\begin{tikzpicture}[scale=1]
	
		\node () at (-2,1){} ;

	\node (a) at (-0.7,0) [circle,fill,inner sep=1pt, outer sep=2pt] {};
	\node (b) at (-0.7,1)[circle,fill,inner sep=1pt , outer sep=2pt]{};
	\node (c) at (-0.7,2) [circle,fill,inner sep=1pt, outer sep=2pt] {};
	\node (a') at (0.5,0) [circle,fill,inner sep=1pt, outer sep=2pt] {};

	\node () at (-1.1,0)[]{$a$};
	\node () at (0.9,0.05)[]{$a'$};
	
	\node () at (-1.1,1)[]{$b$};
	\node () at (-1.1,2)[]{$c$};

	
	\node (al) at (-0.85,-0.1){} ;
	\node (a'l) at (0.38,-0.14){} ;
	\node (bl) at (-0.85,1){};
	\node (bl2) at (-0.86,0.9){};
	\node (cl) at (-0.85,2.1) {};
	
	\node (ar) at (-0.55,-0.1) {};
	\node (a'r) at (0.65,0.1) {};
	\node (br) at (-0.55,1){};
	\node (br2) at (-0.5,1){};
	\node (cr) at (-0.55,2.1){} ;

	\draw	[draw=blue, 	fill=blue,  fill opacity=0.2] plot [smooth cycle] coordinates {  (a'l)(a'r)(br2)(cr) (cl)(bl2)};
	
	\draw	[draw=red, 
	fill=red,  fill opacity=0.2] plot [smooth cycle] coordinates { (al) (ar)(br)(cr)(cl)(bl) };
	
	\draw[->, red, line width = 1pt](a') to (a);
	\node (al) at (0,0){{\LARGE  \color{red}$\times$}} ;


	\draw[->, line width = 1pt](a) to (b);

\end{tikzpicture}

\end{tabular}
	
	\caption[Illustration for the proof of Lemma~\ref{lem:no-cycle}.]{The statement of Lemma~\ref{lem:no-cycle}.	The blue shape means $(a',b,c)\in R$, the crossed-out red arrow means  $(a',a)$ is not a semilattice edge.}\label{fig:no cycle}
\end{figure}

\begin{lemma}\label{lem:no-cycle}
	Let $a, a', b, c \in A_0$ be such that $(a,b,c) \notin R$, $(a,a,b) \notin R$, and $(a',b,c) \in R$. 
{	Then $(a',a)$ is not a semilattice edge.}
\end{lemma}

\begin{proof}
	Assume for contradiction $(a',a)$ is a semilattice edge, i.e., there exists $p\in \Pol(\mathfrak{A}_0) $ with $p(a,a')=a=p(a',a)$. Note that by Lemma~\ref{lem: aaa in R} it follows that $(a,a,a)\in R$ and $(a,b,b)\in R$. 

\medskip

	\noindent\underline{Claim 1:}  $p(b,a)=a$ implies $p(a,b)=b$. This follows immediately, since otherwise \\$p((a,b,b),(b,a,b))=(a,a,b)\in R$ is a contradiction.
	
	\medskip
	
		\noindent\underline{Claim 2:} $(a,a,c)\not \in R$. We assume the opposite and consider the only   two possible cases for $p(b,a)$.
	\begin{enumerate}
		\item $ p(b,a)=b$: We get a contradiction by $p((a',b,c), (a,a,c))=(a,b,c) \in R$.
		\item $ p(b,a)=a$: By Claim 1 we know that $p(a,b)=b$ follows. 
		Then  $p((a,a,c),(a',b,c))=(a,b,c) \in R$ contrary to our assumptions.
		
	\end{enumerate}
	This proves Claim 2. 
	
	\medskip

			\noindent\underline{Claim 3:}  $p(c,a)=a$ implies $p(a,c)=c$. Lemma~\ref{lem: aaa in R} together with Claim 2 implies that $(a,c,c)\in R$. Now Claim 3 follows immediately, since otherwise $p((a,c,c),(c,a,c))=(a,a,c)\in R$, which contradicts Claim 2.
	
	\medskip

	We finally make a case distinction for all possible values of $p$ on $(b,a)$ and $(c,a)$.
	
	\begin{enumerate}
		\item $ p(b,a)=b$ and $p(c,a)=c$: We get a contradiction  by $p((a',b,c), (a,a,a))=(a,b,c) \in R$.
		\item $ p(b,a)=b$ and $p(c,a)=a$: We get a contradiction by $p((a',b,c), (a,a,a))=(a,b,a) \in R$.
		
		\item $ p(b,a)=a$ and $p(c,a)=c$:  $p((a',b,c),(a,a,a))=(a,a,c) \in R$ contradicts Claim 2.
		\item $ p(b,a)=a$ and $p(c,a)=a$: By Claim 1 we get $p(a,b)=b$ and by Claim 3 we get $p(a,c)=c$. This yields a contradiction by $p((a,a,a),(a',b,c))=(a,b,c) \in R$.
	\end{enumerate}
This proves the lemma.
\end{proof}


\subsection{The Binarisation}\label{sec:bina}
We have announced in the introduction that we want to apply Kazda's theorem (Theorem~\ref{thm:kazda})  for binary conservative structures, but the atom structure $\fAo$ from Section~\ref{sect:atomstruct} has a ternary relation. We therefore associate a certain binary structure $\fAob$ to $\fAo$ which shares many properties with $\fAo$.

\begin{definition}\label{defi:binarisation}
	We denote by 
	$\fAob$ the structure with domain $A_0$ and the following relations:
	\begin{itemize}
		\item a unary relation $U_S$ for each subset $S$ of $A_0$; 
		\item for every $a \in A_0$ the  binary relation $R_a := \{(x,y)\in A_0^2\mid (a,x,y)\in R\}$;
		\item 
  a binary relation for every union of relations of the form $R_a$.
	\end{itemize}
	
\end{definition}

The binarisation of $\fAo$ according to Definition~\ref{defi:binarisation} will be denoted by 
$\fAob$. 
We obtain the following results about the relationship of 	$\Pol(\fAo)$ and $\Pol(\fAob)$.

\begin{lemma}\label{lem: Polinclusion}
	$\Pol(\mathfrak{A}_0) \subseteq \Pol(\mathfrak{A}_0^b)$.
\end{lemma}
\begin{proof}
	Clearly, every relation $R_a$ has the primitive positive definition
	$\exists z (U_{\{a\}}(z) \wedge R(z,x,y))$ in $\fAo$. 
	A primitive positive definition of $\cup_{a \in S} R_a$ is
	$\exists z (U_{S}(z) \wedge R(z,x,y))$ in $\fAo$.  Then the statement of the lemma follows by Theorem~\ref{theo:BKKR}.
\end{proof}

\begin{lemma}\label{lem:p2}
	$\Pol^{(2)}(\mathfrak{A}_0^b) \subseteq \Pol^{(2)}(\mathfrak{A}_0)$.
\end{lemma}

\begin{proof}
	Let $f \in \Pol^{(2)}(\fAob)$. It suffices to prove that $f$ preserves the relation  $R^{\fA_0}$. Arbitrarily choose $(a_1,b_1,c_1),(a_2,b_2,c_2) \in R$. We want to show that 
	$t := (f(a_1,a_2),f(b_1,b_2),f(c_1,c_2))$ is in $R$ as well.  
	If $t \in \{(a_1,b_1,c_1),(a_2,b_2,c_2)\}$ then there is nothing to be shown. Otherwise, since $f$ must preserve $\{a_1,a_2\}$, $\{b_1,b_2\}$, and $\{c_1,c_2\}$, 
	by the symmetry of $R$ and possibly flipping the arguments of $f$ we may assume without loss of generality that $f(a_1,a_2) = a_1$, $f(b_1,b_2)=b_1$, and $f(c_1,c_2) = c_2$. 
	So we have to show that $t = (a_1,b_1,c_2) \in R$. 
	Note
	that $(b_1,c_1) \in R_{a_1}$ and $(b_2,c_2) \in R_{a_2}$, 
	and therefore $(f(b_1,b_2),f(c_1,c_2)) \in R_{a_1} \cup R_{a_2}$. If $(f(b_1,b_2),f(c_1,c_2)) \in R_{a_1}$, then 
	we obtain that 
	$(b_1,c_2) \in R_{a_1}$, and hence $(a_1,b_1,c_2) \in R$ and we are done.  Otherwise, $(f(b_1,b_2),f(c_1,c_2)) \in R_{a_2}$, and we obtain that $(b_1,c_2) \in R_{a_2}$, and hence $(a_2,b_1,c_2) \in R$. 
	In particular, $(a_2,c_2) \in R_{b_1}$. 
	Since  $(a_1,c_1) \in R_{b_1}$ 
	and since $f$ preserves $R_{b_1}$ we have that 
	$(f(a_1,a_2),f(c_1,c_2)) = (a_1,c_2) \in R_{b_1}$,
	and hence $(a_1,b_1,c_2) \in R$, which concludes the proof. 
\end{proof}
Observe that this implies that $\fAob$ and $\fA$ have exactly the same semilatice edges. The following example shows
that in general it does not hold
that $\Pol(\fAob)\subseteq \Pol(\fAo)$. 

\begin{example}\label{examp:Komega}
	Let $\mathbf{K}$ be the relation algebra with two atoms $\{\id,E\}$ and the multiplication table given in Figure~\ref{Table:komega} (called $1_2$ in the terminology of Maddux~\cite{Maddux2006-dp}). It is easy to see that the expansion of the infinite clique $K_\omega$  
	by the empty relation for $0$, the equality relation for $\id$, and the full relation for $1$ is a normal representation of $\mathbf{K}$.
	It is not hard to show that $\mathfrak{K}_0$ does not have a majority polymorphism, but $\mathfrak{K}^b_0$  does since
	every binary relation on a two-element set is preserved by the (unique) majority operation on a two-element set.

	\begin{figure}[t]
		\begin{center}
			
			\begin{tabular}{|c||c|c|}
				\hline 
				$~\circ~$	& $~\id~$ & $E$ \\ 
				\hline \hline
				$\id$	&  $\id$ & $E$  \\ 
				\hline 
				$E$	&$E$  & $ 1$  \\ 
				\hline 
				
			\end{tabular}

		\end{center}

		\caption{Multiplication table of  the relation algebra $\mathbf{K}$.}
		\label{Table:komega}
	\end{figure}

\end{example}

\subsection{No Affine Edges in the Atom Structure}\label{sec: no affine}

We show in this section that under the assumption that $\fAob$ has a Siggers polymorphism and has no affine edge, $\fAo$ also has no affine edge. So let us assume for the whole section that $\Pol(\fAob)$ contains a Siggers operation and that $\fAob$ has no affine edge. 

 Since $\fAob$ is conservative and has no affine edge, there exists according to Proposition~\ref{prop:bulatov color functions}  a binary operation $v \in \Pol(\fAo)$ and a ternary operation $w\in \Pol(\fAo)$ such that for every two element subset $C$ of $A_0$,
 
 \begin{itemize}
 	\item $v|_{C}$ is a semilattice operation whenever $C$ has a semilattice edge, and $v|_{C}(x,y)=x$ otherwise;
 	\item $w|_{C}$ is a majority operation if $C$ is  a majority edge and $w|_{C}(x,y,z)=v|_{C}( v|_{C}(x,y),z)$ if $C$ has a semilattice edge.
 \end{itemize}
We define 
\begin{align}
u(x,y,z):=w(v (v (x, y), z), v (v (y, z), x), v (v (z, x), y)). \label{eq:comp}
\end{align}

{
\begin{lemma}\label{lem:only sl of majority}The structures $\fAo$ and $\fAob $ have exactly the same semilattice edges.
	Let $a,b\in A_0$ be such that $\{a,b\}$ has no semilattice edge in 
 the two
 structures. Then the restriction of $u$ to $\{a,b\}$ is a majority operation.
	\end{lemma}

\begin{proof}By Lemma~\ref{lem: Polinclusion} and Lemma~\ref{lem:p2}, the structures $\fAo$ and $\fAob $ have exactly the same semilattice edges, since they have the same binary polymorphisms. The second statement follows  from the definition of $u$ by  means of $w$ and $v$.
	\end{proof}
}

\begin{definition}
	Let $f$ be a binary operation on $A_0$. Then we say that $\{a,b,c\}\subseteq A_0$ has the \emph{$f$-cycle $(x,y,z)$} if  $\{x,y,z\}=\{a,b,c\}$ and $(x,y)$, $(y,z)$, and $(z,x)$ are $f$-sl.
	
\end{definition}

\begin{lemma}\label{lem:cyclic} 
	Let $a,b,c\in A_0$ be  such that $(a,b)$ is  $v$-sl  
	but $(a,b,c)$ is not a $v$-cycle. Then $u(r,s,t)\not = a$ for any choice of $r,s,t \in A_0$ such that $\{r,s,t\}=\{a,b,c\}$.
\end{lemma}

\begin{proof}
	We prove a series of intermediate claims.

	\medskip 
\noindent	\underline{Claim 1:} If $\{x,y,z\}=\{a,b,c\}$ and $v (v (x, y), z)=a$, then $z=a$. 
	
\medskip 
We assume for contradiction that $z \neq a$ and distinguish the following cases. 
	\begin{enumerate}
		\item $x=a, y=b, z=c$: Then $ v (v (x, y), z)= v (b, c) \in \{b,c\}$.
		\item $x=a, y=c, z=b$: Then $ v (v (x, y), z)\in \{ v(a,b), v(c,b)\} \subseteq \{b,c\}$.
		\item $x=b, y=a, z=c$: Then $ v (v (x, y), z)= v (b, c) \in \{b,c\}$.
		\item $x=c, y=a, z=b$: Then $ v (v (x, y), z)\in \{ v (c, b), v(a,b)\} \subseteq \{b,c\}$.
	\end{enumerate}
	In all four cases we have $v (v (x, y), z) \neq a$, which contradicts our assumption and proves the claim. 
	
	

	\medskip

\noindent	\underline{Claim 2:} If $\{x,y,z\}=\{a,b,c\}$ and $v (v (x, y), z)=a$, then $(c,a)$ is $v$-sl.
	
	By Claim 1 we get that $z=a$ and thus $\{x,y\} = \{b,c\}$. We then have $v(x,y)=c$ since otherwise
	$v(x,y) = b$ and $v (v (x, y), z) = v(b,a) =b$, which contradicts our assumption.
	Assume for contradiction that $(c,a)$ is not $v$-sl and therefore one of the following holds:
	
	\begin{enumerate}
		\item $(a,c)$ is $v$-sl.  It follows that $v (v (x, y), z)= v(c,a)=c$ which contradicts our assumption.
		\item $\{a,c\}$ is a majority edge of $\fAob$. It follows again that $v (v (x, y), z)= v(c,a)=c$,
		since $v$ behaves like the projection on the first coordinate on majority edges. This contradicts our assumption.
	\end{enumerate}
	
	\noindent\underline{Claim 3:} If $\{x,y,z\}=\{a,b,c\}$ and $v (v (x, y), z)=a$, then $\{b,c\}$ is a majority edge of $\fAob$.

	Assume for contradiction that
	there is a semilattice edge on $\{b,c\}=\{x,y\}$. 
	By Claim 2 and our assumption that $(a,b,c)$ is not a $v$-cycle,  the edge $(b,c)$ is not $v$-sl and therefore $(c,b)$ is  $v$-sl. Therefore, we get $v(v(x,y),z)=v(b,a)=b$ which contradicts our assumption.
	
	\medskip 
\noindent	\underline{Claim 4:} If $\{x,y,z\}=\{a,b,c\}$ and $v (v (x, y), z)=a$, then $v (v (z, x), y) = b=v (v (y, z), x)$ follows. 
	
	By Claim 3, $\{b,c\}$ is a majority edge of $\fAob$  and it follows that $b=y$ and $c=x$ since otherwise $v (v (x, y), z)=v(b,a)=b$.
	Now we calculate 
	$$ v (v (z, x), y))= v (v (a, c), b))=v(a,b)=b=v (b , c) =v (v (b, a) , c) =v (v (y, z) , x)$$ which proves the claim.

	\medskip
	
	Now we are able to prove the statement of the lemma. Assume for contradiction that $u(r,s,t)=a$. Since $w$ preserves $U_{A \setminus \{a\}}$ this is only possible if at least one of the terms $v (v (r, s), t)$, $v (v (s, t), r)$, or $ v (v (t, r), s))$ evaluates to $a$.
	By Claim 4 we get that the two other terms evaluate to $b$. Since $(a,b)$ is $v$-sl we get that $w(a,b,b)=w(b,a,b)=w(b,b,a)=b$ which contradicts our assumption $u(r,s,t)=a$.
\end{proof}

\begin{theorem}\label{theo: u from left to right}
	$\Pol(\mathfrak{A}_0^b) \subseteq \Pol(\mathfrak{A}_0)$.
\end{theorem}
\begin{proof}We have to prove that $u$ preserves $R$.
	Let $(a_1,b_1,c_1), (a_2,b_2,c_2),  (a_3,b_3,c_3)\in R$ and
	let $$(a,b,c) := (u(a_1,a_2,a_3), u(b_1,b_2,b_3), u(c_1,c_2,c_3)).$$
	Assume for contradiction  that $(a,b,c) \not\in R$.  
	By Lemma~\ref{lem:taylor}, we may 
	assume without loss of generality that $(a,a,b) \not \in R$ and hence 
	by Lemma~\ref{lem:aabforb} $(a,b)$ is a semilattice edge in $\fAo$ and $(b,a)$ is not.
 {By Lemma~\ref{lem:only sl of majority} the structures $\fAo$ and $\fAob $ have exactly the same semilattice edges. This implies that $\{a,b\}$ has a semilattice edge; this semilattice edge can only be $(a,b)$ and therefore $(a,b)$ is $v$-sl.} 
	Since $u$ preserves
	$R_{a_1} \cup R_{a_2} \cup R_{a_3}$ 
	there exists $r \in \{a_1,a_2,a_3\}$ such that 
	$(r,b,c) \in R$. 
	By Lemma \ref{lem:no-cycle} we get that $(r,a)$ is not a semilattice edge in  $\mathfrak{A}_0$ and therefore  Lemma~\ref{lem:p2} implies that  $(r,a)$ is not a semilattice edge in $\fAob$ and we get that $(r,a)$ is not  $v$-sl. 
	Let $s \in \{a_1,a_2,a_3\}\setminus \{a,r\}$. 
	\medskip
	
	\noindent\underline{Claim 1:} $\{a,r,s\}$ does not have a $v$-cycle.  
	
	Assume for contradiction that $\{a,r,s\}$ has a $v$-cycle.  Since $(r,a)$ is not $v$-sl it follows that $(a,r)$ is $v$-sl and therefore
	$(r,s)$ and $(s,a)$ are $v$-sl. 
	We consider the following two cases:
	\begin{enumerate}
		\item $(s,b,c)\in R$. Then Lemma~\ref{lem:no-cycle} applied to $a,s,b,c$ implies that $(s,a)$ is not a semilattice edge and therefore by Lemma~\ref{lem:p2} $(s,a)$ is not  $v$-sl, which is a contradiction. 
		
		\item  $(s,b,c)\not\in R$. Note that $(s,s,b) \not\in R$ holds,  since  $(s,a)$ is $v$-sl and $v((s,s,b), (a,a,a))=(a,a,b)\in R$ yields a contradiction to $(a,a,b) \not\in R$.
			Hence, Lemma \ref{lem:no-cycle} applied to $s,r,b,c$
			implies that $(r,s)$ is not a semilattice edge and therefore by Lemma~\ref{lem:p2} $(r,s)$ is not  $v$-sl, which is again a contradiction.
		
	\end{enumerate}

This proves that  $\{a,r,s\}$ cannot have a $v$-cycle. 
	
	\medskip 
	\noindent\underline{Claim 2:}  $u(a,a,r)=a$. 
	
	Assume for contradiction that $u(a,a,r)=r$.
	Then $\{a,r\}$ is clearly not a majority edge of $\fAob$, and 
	since $\fAob$ does not have affine edges it  follows that $(a,r)$ is $v$-sl. Furthermore,  $(a,r,s)$ is not a $v$-cycle and therefore Lemma~\ref{lem:cyclic} implies that $u(a_1,a_2,a_3)\not =a$ which contradicts  the definition of $a$. 
	
	\medskip
	
	
	Finally, consider the following application of the polymorphism $u$:
	$$u \big( (a,a), (a,a) ,(r,b) \big) = (a, b). $$
	Since $(a,a) \in R_{a}$ and $(r,b) \in R_c$ and since $u$ is in $\Pol(\mathfrak{A}_0^b)$ we get that  $(a,b)\in R_a \cup R_c$. Hence, ${(a,a,b)} \in R$ or $(c,a,b)\in R$, which contradicts our assumptions.
\end{proof}

\subsection{Proof of the Main Theorem}
We can now prove the main result of this section.

\begin{proof}[Proof of Theorem~\ref{theo:Datatractability}.]
		Let $\bA$ be a finite relation algebra that satisfies the assumptions of Theorem~\ref{theo:Datatractability} and let $\fAo$ be the atom structure of $\bA$ (Definition~\ref{def: atom str}). We denote by $\fAob$ the binarisation of $\fAo$ according to Definition~\ref{defi:binarisation}. It follows from the assumptions on $\bA$ that $\Pol(\fAo)$ contains a Siggers operation. By Lemma~\ref{lem: Polinclusion} we get that $\Pol(\fAob)$ contains a Siggers operation as well. Note that $\fAob$ is a finite binary conservative structure and therefore Theorem~\ref{thm:kazda} implies that $\fAob$ has no affine edges. Therefore, $\fAob$  satisfies the general assumption from  Section~\ref{sec: no affine} and we can define the operation $u$ as in~\eqref{eq:comp}. 
		Note that $u$ witnesses by Lemma~\ref{lem:only sl of majority} that $\fAob$ does not have an affine edge. We can now apply Theorem~\ref{theo: u from left to right} and get that $u$ is also a polymorphism of $\fAo$. Recall that $\fAo$ and $\fAob$ have by Lemma~\ref{lem:p2}  exactly the same semilattice edges and therefore 
	Lemma~\ref{lem:only sl of majority} and the fact that $u$ is a polymorphism of $\fAo$ imply that $\fAo$ does not have an affine edge. By Proposition~\ref{prop:no affine implies bounded w} we get that there exists a $3$-ary weak near unanimity polymorphism $f\in \Pol(\fAo)$ and a $4$-ary weak near unanimity polymorphism $g\in \Pol(\fAo)$ such that $$	\forall x,y,z \in B.~ f(y,x,x)=g(y,x,x,x)$$
	holds.  Theorem~\ref{theo: datalog reduction} implies that $\Csp(\fB)$ and thus also $\NSP(\bA)$ can be solved by $(4,6)$-consistency algorithm. \end{proof}

\section{Consistency and Symmetric Flexible-Atom Algebras}
We apply our result from Section~\ref{sec: LocalCons} to the class of finite symmetric relation algebras with a flexible atom and obtain a $k$-consistency versus NP-complete complexity dichotomy. 

A finite relation algebra  $\bA$ is called  \emph{integral} if the element $\id$ is an atom of $\bA$, i.e., $\id\in A_0$. We define flexible atoms for integral relation algebras only. For a discussion about integrality and flexible atoms consider Section 3 in \cite{BodKnaeFlexJournal}.

\begin{definition}
Let $\mathbf{A} \in \ra$ be finite and integral.
		An atom $s\in A_0$ is called \emph{flexible} if for all $a,b \in A \setminus \{\id\}$ it holds that $s\leq a\circ b$.
	\end{definition}

Relation algebras with a flexible atom have been studied intensively in the context of the \emph{flexible atoms conjecture}~\cite{Maddux1994,Alm2008}. It can be shown easily that  finite relation algebras with a flexible atom  have a {normal representation} \cite{BodirskyKnaeAAAI,BodKnaeFlexJournal}. In  \cite{BodKnaeFlexJournal} the authors obtained a P versus NP-complete complexity dichotomy for NSPs of finite symmetric relation algebras with a flexible atom (assuming P $\neq$ NP). In the following we strengthen this result and prove that every problem in this class can be solved by $k$-consistency for some $k\in \mathbb{N}$ or is NP-complete (without any complexity-theoretic assumptions).

We combine Theorem~\ref{theo:Datatractability} with the main result of~\cite{BodKnaeFlexJournal} to  obtain the following characterization for NSPs of finite symmetric 
relation algebras with a flexible atom that are solved by the $(4,6)$-consistency procedure. {Note that the difference of Theorem~\ref{theo:Datalogcharact} and the related result in~\cite{BodKnaeFlexJournal} is the algorithm that solves the problems in P.}
	
	\begin{theorem}\label{theo:Datalogcharact} Let $\bA$ be a  finite symmetric integral relation algebra with a flexible atom.
		Then the following are equivalent:
		\begin{itemize}
			\item $\bA$ admits a Siggers behavior.
	\item $\Nsp(\mathbf{A})$ can be solved by the $(4,6)$-consistency procedure.
	\end{itemize}
	\end{theorem}
	\begin{proof}
Every finite symmetric relation algebra $\bA$ with a flexible atom has a normal representation $\fB$ by Proposition 3.5 in \cite{BodKnaeFlexJournal}. 

If the first item holds it follows from Proposition 6.1.~in \cite{BodKnaeFlexJournal} that $\fB$ has a binary injective polymorphism. By Lemma~\ref{lem:injective implies cycle} the relation algebra $\bA$ has all $1$-cycles. 
We apply Theorem~\ref{theo:Datatractability} and get that the second item in Theorem~\ref{theo:Datalogcharact} holds.
 
 
 We prove the converse implication by showing the contraposition. Assume that the first item is not satisfied. 
 {Then Theorem 9.1 in \cite{BodKnaeFlexJournal} implies that there exists a polynomial-time reduction from $\csp(K_3)$ to $\Nsp(\mathbf{A})$ which preserves solvability by the $(k,l)$-consistency procedure. The problem $\csp(K_3)$ is the 3-colorability problem which is known {(e.g., by \cite{BartoKozikFOCS09})} to be not solvable by the $(k,l)$-consistency procedure for every $k,l\in \mathbb{N}$. Hence $\Nsp(\mathbf{A})$ cannot be solved by the $(4,6)$-consistency procedure. }\end{proof}


	As a consequence of Theorem~\ref{theo:Datalogcharact} we obtain the following strengthening of the complexity dichotomy NSPs of finite symmetric integral relation algebra with a flexible atom~\cite{BodKnaeFlexJournal}.

\begin{corollary}[Complexity Dichotomy]
\label{cor:dicho} 
Let $\bA$ be a finite symmetric integral relation algebra  with a flexible atom. Then $\NSP(\bA)$ can be solved by the $(4,6)$-consistency procedure, or it is NP-complete.
	\end{corollary}

\begin{proof}
Suppose that the first condition in Theorem~\ref{theo:Datalogcharact} holds. Then Theorem~\ref{theo:Datalogcharact} implies that $\NSP(\bA)$ can be solved by the $(4,6)$-consistency procedure. If the first condition in Theorem~\ref{theo:Datalogcharact} is not satisfied, then it follows from Theorem 9.1 in \cite{BodKnaeFlexJournal} that $\NSP(\bA)$ is NP-complete.
	\end{proof}

\section{The Complexity of the Meta Problem}
In this section we study the computational complexity of deciding for a given finite symmetric relation algebra ${\bf A}$ with a flexible atom whether 
the $k$-consistency algorithm solves $\nsp({\bf A})$. We show that this problem is decidable in polynomial time even if ${\bf A}$ is given 
by the restriction of its composition table to the atoms of ${\bf A}$: note that this determines a symmetric relation algebra  uniquely, and that this is an (exponentially) more succinct representation of ${\bf A}$ compared to explicitly storing the full composition table.


\begin{definition}[Meta Problem]
We define $\Meta$ as the following computational problem. \\
{\bf Input:} the composition table of a finite symmetric relation algebra ${\bf A}$ restricted to $A_0$. \\
{\bf Question:} is there a $k \in {\mathbb N}$ such that $k$-consistency solves $\nsp({\bf A})$? 
\end{definition}

\begin{restatable}{theorem}{thmmeta}\label{thm:meta}
The problem $\Meta$ is undecidable. However, it 
can be decided in polynomial time
if the input is restricted to finite symmetric integral relation algebras ${\bf A}$ with a flexible atom. 
\end{restatable}

\begin{proof} 
The undecidability for Meta follows from Theorem~\ref{thm:RBCP-undecidable}. 
If $\bA$ has a flexible atom, then by Theorem~\ref{theo:Datalogcharact} it suffices to test the existence of an operation $f \colon A^6_0 \to A_0$ which satisfies conditions 1.-3.\ in Definition~\ref{defi:sigersbeh}. The three conditions can clearly be checked in polynomial time, so we already know that $\Meta$ is in NP.

Note that the search for $f$ may be phrased as an instance of $\csp(\fA_0)$ with $|A|^6$ variables (i.e., we add constraints that force any solution to be a polymorphism, and we add additional equality constraints to force that any solution will be a Siggers operation). Using the fact that the $k$-consistency procedure is one-sided correct even in the case that $\csp(\fA_0)$ is NP-hard  (i.e., if the procedure rejects a given instance of $\csp(\fA_0)$, then the instance is always unsatisfiable), 
we may use a standard self-reducibility argument to obtain a polynomial-time algorithm for finding $f$ (see, e.g.,~\cite{MetaChenLarose}; the idea is to run the $k$-consistency algorithm on the instance described above; if the answer `no', then clearly $\bA$ has no Siggers behaviour. If the answer is `yes', we additionally add a constraint that forces some variable $x$ to some value $a \in A$. If the resulting instance is unsatisfiable for all possible $a \in A$, then $k$-consistency didn't give the correct answer at the first place, and by our result $\bA$ has no Siggers behaviour. If for some $a \in A$, the resulting instance is satisfiable, we continue in this fashion until each variable is assigned to some value from $A$, which then provides a solution to the instance. Therefore, in this case $\bA$ has a Siggers behavior.)
\end{proof}

\section{Conclusion and Open Questions}
The question whether the network satisfaction problem for a given finite relation algebra can be solved by the famous $k$-consistency procedure is undecidable. Our proof of this fact heavily relies on prior work of Hirsch~\cite{Hirsch-Undecidable} and of Hirsch and Hodkinson~\cite{HirschHodkinsonRepresentability} and shows that almost any question about the network satisfaction problem for finite relation algebras is undecidable. 

However, if we further restrict the class of finite relation algebras, one may obtain strong classification results. We have demonstrated this for the class of finite symmetric integral relation algebras with a flexible atom (Theorem~\ref{cor:dicho});
the complexity of deciding whether the conditions in our classification result hold drops from undecidable to P (Theorem~\ref{thm:meta}). 
One of the remaining open problems is a characterisation of the power of $k$-consistency for the larger class of all finite relation algebras with a normal representation.

Our main result (Theorem~\ref{theo:Datatractability}) is a sufficient condition for the applicability 
of the $k$-consistency procedure; the  condition does not require the existence of a flexible atom but applies more generally to finite symmetric relation algebras $\bf A$ with a normal representation. Our condition consists of two parts: the first is the existence of all 1-cycles in $\bf A$, the second is that $\bf A$ admits a Siggers behavior. We conjecture that dropping the first part of the condition leads to a necessary and sufficient condition for solvability by the $k$-consistency procedure
in the case that the normal representation of $\bf A$ has a primitive automorphism group. 

\begin{conjecture}\label{conj:wild}
Let $\bf A$ be a finite symmetric relation algebra with a normal representation whose automorphism group is primitive. 
Then $\bf A$ 
admits a Siggers behavior if and only if 
$\nsp(\mathbf{A})$ can be solved by the $k$-consistency procedure for some $k\in \mathbb{N}$.
\end{conjecture}

Note that this conjecture generalises Theorem~\ref{theo:Datalogcharact},
because the normal representation of a relation algebra with a flexible atom always has a primitive automorphism group. 
The forward direction of the conjecture is true, 
because if $\bA$ admits a Siggers behavior,  
it must have 
all 1-cycles~\cite{BodirskyKnaeRamics},  and hence the claim follows from our main result           (Theorem~\ref{theo:Datalogcharact}). 

Clearly, the assumption that $\bA$ is symmetric cannot be dropped, since the point algebra, which has a normal representation with a primitive automorphism group, can be solved by the path consistency procedure, but does not admit a Siggers behavior. The following example shows that we cannot drop the primitivity condition in our conjecture as well.

\begin{figure}[t]
	\begin{center}
		
		\begin{tabular}{|c||c|c|c|}
			\hline 
			$~\circ~$	& $~\id~$ & $E$ & $N$ \\ 
			\hline \hline
			$\id$	&  $\id$ & $E$& $N$  \\ 
			\hline 
			$E$	&$E$  & $ \id$  & $ N$  \\ 
			\hline 
			$N$	&$N$  & $ N$ & $ 1$   \\ 
			\hline 
			
		\end{tabular}

	\end{center}

	\caption{Multiplication table of  the relation algebra $\mathbf{C}$.}
	\label{Table:omegaK2}
\end{figure}

\begin{example}
\label{expl:omegaK2}
Consider the relation algebra $\mathbf{C}$ with atoms $\{\id,E,N\}$ and the multiplication table in Figure~\ref{Table:omegaK2} (called $3_7$ in the terminology of Maddux~\cite{Maddux2006-dp}).  This relation algebra has a normal representation, namely the expansion of the infinite disjoint union of the clique $K_2$ by all first-order definable binary relations. We denote this structure by $\overline{\omega K_2}$. 
One can observe that $\Csp(\overline{\omega K_2})$ and therefore also the NSP of the relation algebra can be solved by the  $(2,3)$-consistency algorithm (for details see \cite{KnPhD}).

The normal representation of $\bf C$ has an imprimitive automorphism group ($\mathbf{C}$ does not have all $1$-cycles),
and therefore does not fall into the scope of  Theorem~\ref{theo:Datatractability}. 
In fact, our proof of Theorem does not work for $\mathbf{C}$, because the CSP of the atom structure $\fC_0$ of $\bC$ cannot be solved by the $k$-consistency procedure for some $k\in \mathbb{N}$. Hence, the reduction of  $\nsp(\mathbf{C})$ to  $\csp(\fC_0)$ (incorporated in Theorem~\ref{theo: datalog reduction}) does not imply that $\nsp(\mathbf{C})$ can be solved by $k$-consistency procedure for some $k\in \mathbb{N}$. Indeed, it can be shown that $\bC$ does not admit a Siggers behavior (the existence of a Siggers behavior has recently been verified for all relation algebras with at most 4 atoms with the help of a computer~\cite{bodirsky2025networksatisfactionproblemrelation}; however, the computation for this specific algebra can also be carried out easily by hand). 
\end{example}

It would be interesting to characterise the power of the $k$-consistency procedure for the NSP of finite relation algebras with a normal representation whose automorphism group is imprimitive. In this case, there is a non-trivial definable equivalence relation. It is already known that if this equivalence relation has finitely many classes, then the NSP is NP-complete and the $k$-consistency procedure does not solve the NSP~\cite{BodirskyKnaeRamics}. Similarly, the NSP is NP-complete if there are equivalence classes of finite size larger than two. 
    It therefore remains to study the case of     
    infinitely many two-element classes, and with infinitely many infinite classes. In both cases we wish to reduce the classification to the situation with a primitive automorphism group.

Finally, we ask whether it is true that if $\bf A$ is a finite symmetric relation algebra with a flexible atom and $\nsp(\bf A)$ can be solved by the $k$-consistency procedure for some $k$, then it can also be solved by the $(2,3)$-consistency procedure? In other words, can we improve $(4,6)$ in Corollary~\ref{cor:dicho} to $(2,3)$?  

\section*{Acknowledgements}
The authors thank Moritz Jahn and Paul Winkler for their feedback on the journal version of the article. 

\bibliography{global}

\end{document}